\documentclass{article}
\usepackage{color, xcolor, colortbl}
\usepackage{graphicx,epstopdf}
\usepackage{geometry}
\usepackage{amsmath,amssymb,amsthm}
\usepackage{algorithm}
\usepackage{algorithmic}
\usepackage{bm}
\usepackage[caption=false]{subfig}
\usepackage{appendix}
\usepackage{multirow}
\usepackage{braket}
\usepackage[english]{babel}
\usepackage{hyperref}
\usepackage[capitalize]{cleveref}
\usepackage[square,numbers]{natbib}
\usepackage{adjustbox}
\usepackage{xspace}
\usepackage[roman]{complexity}

\newcommand{\diag}{\operatorname{diag}}

\newcommand{\Tr}{\operatorname{Tr}}

\newcommand{\I}{\mathrm{i}}

\newcommand{\wt}[1]{\widetilde{#1}}

\newcommand{\norm}[1]{\left\lVert#1\right\rVert}

\newcommand{\Or}{\mathcal{O}}

\newcommand{\CC}{\mathbb{C}}

\newtheorem{thm}{\protect\theoremname}
\theoremstyle{plain}
\newtheorem{lem}[thm]{\protect\lemmaname}
\theoremstyle{plain}

\theoremstyle{plain}
\newtheorem*{lem*}{\protect\lemmaname}
\theoremstyle{plain}
\newtheorem{prop}[thm]{\protect\propositionname}
\theoremstyle{plain}

\providecommand{\definitionname}{Definition}
\providecommand{\assumptionname}{Assumption}
\providecommand{\corollaryname}{Corollary}
\providecommand{\lemmaname}{Lemma}
\providecommand{\propositionname}{Proposition}
\providecommand{\remarkname}{Remark}
\providecommand{\theoremname}{Theorem}

\usepackage{url}
\usepackage{enumitem}
\usepackage{indentfirst}

\def\le{\leqslant}
\def\ge{\geqslant}

\newcommand{\eps}{\epsilon}
\newcommand{\veps}{\varepsilon}

\numberwithin{equation}{section}

\begin{document}

\title{Parallel transport dynamics for mixed quantum states \\ with applications to time-dependent density functional theory}

\author{
Dong An\thanks{
        Department of Mathematics, University of California, Berkeley,  CA 94720, USA (dong\_an@berkeley.edu)
        }, ~~
Di Fang\thanks{%
        Department of Mathematics, University of California, Berkeley,  CA 94720, USA (difang@berkeley.edu)},~~
Lin Lin\thanks{
        Department of Mathematics and Challenge Institute for Quantum Computation, University of California, Berkeley and Computational Research Division, Lawrence Berkeley National Laboratory, Berkeley, CA 94720, USA (linlin@math.berkeley.edu)
        }.
}

\date{\today}

\maketitle
\begin{abstract}
Direct simulation of the von Neumann dynamics for a general (pure or mixed) quantum state can often be expensive. 
One prominent example is the real-time time-dependent density functional theory (rt-TDDFT), a widely used framework for the first principle description of many-electron dynamics in chemical and materials systems.
Practical rt-TDDFT calculations often avoid the direct simulation of the von Neumann equation, and solve instead a set of Schr\"odinger equations, of which the dynamics is equivalent to that of the von Neumann equation.
However, the time step size employed by the Schr\"odinger dynamics is often much smaller. In order to improve the time step size and the overall efficiency of the simulation, we generalize a recent work of the parallel transport (PT) dynamics for simulating pure states [An, Lin, Multiscale Model. Simul. 18, 612, 2020] to general quantum states.
The PT dynamics provides the optimal gauge choice, and can employ a time step size comparable to that of the von Neumann dynamics.
Going beyond the linear and near adiabatic regime in previous studies, we find that the error of the PT dynamics can be bounded by certain commutators between Hamiltonians, density matrices, and their derived quantities. Such a commutator structure is not present in the Schr\"odinger dynamics.
We demonstrate that the parallel transport-implicit midpoint (PT-IM) method is a suitable method for simulating the PT dynamics, especially when the spectral radius of the Hamiltonian is large.
The commutator structure of the error bound, and numerical results for model rt-TDDFT calculations in both linear and nonlinear regimes, confirm the advantage of the PT dynamics. 
\end{abstract}

\section{Introduction}

Consider the problem of solving a finite dimensional, (possibly) nonlinear von Neumann equation 
\begin{equation} 
\label{eq:vonN}
\I \partial_t \rho(t) =  [H(t, \rho(t)), \rho(t)], \quad \rho(0)=\rho_0,
\end{equation}
where $\rho_0\in\CC^{N_g\times N_g}$ is a Hermitian matrix  satisfying $\rho^2_0\preceq \rho_0$. 
Here $[A,B]=AB-BA$ is the commutator of $A$ and $B$, and $A\preceq B$ means that $A-B$ is a negative semidefinite matrix. The initial quantum state $\rho_0$ is called a pure state if $\rho_0^2=\rho_0$, and a mixed state if $\rho_0^2\prec \rho_0$.
\cref{eq:vonN} can be used to describe the dynamics of a closed quantum system in a very general setting, and we allow the time-dependent Hamiltonian $H(t, \rho(t))\in \CC^{N_g\times N_g}$ to have a nonlinear dependence on entries of $\rho(t)$. 
One prominent application is the real-time time-dependent density functional theory (rt-TDDFT)~\cite{RungeGross1984,YabanaBertsch1996,OnidaReiningRubio2002,Ullrich2011}, which is one of the most widely used techniques for studying ultrafast
properties of electrons, and has resulted in a variety of applications in quantum physics, chemistry, and materials science. 

In practice, $\rho_0$ is often of low-rank, or can be very well approximated by a low-rank matrix. 
For simplicity, let the rank of $\rho_0$ be denoted by $N$ and assume $N\ll N_g$. The von Neumann dynamics neglects the low-rank structure and propagates $\rho(t)$ directly as a dense matrix. 
For rt-TDDFT calculations with a fine discretization scheme (e.g. the planewave basis set, or the finite difference discretization), $N_g$ can be $10^{6}$  or larger, and the direct propagation of \cref{eq:vonN} becomes extremely expensive. 
In this case, the von Neumann dynamics is often replaced by a set of  nonlinear Schr\"odinger equations (see \cref{eqn:tddft}), and the simulation variables become the electron wavefunctions described by a much smaller matrix $\Psi(t)\in\CC^{N_g\times N}$. However, such a rank reduction can come at the cost of the time step size, denoted by $h$. In many applications, $h$ for the von Neumann dynamics~\eqref{eq:vonN} can be chosen to be at the sub-femtosecond scale (1 fs$ = 10^{-15}$ s), while $h$ for the Schr\"odinger dynamics needs to be sub-attosecond scale (1 as$ = 10^{-18}$ s)~\cite{CastroMarquesRubio2004,SchleifeDraegerKanaiEtAl2012,GomezMarquesRubioEtAl2018}. 

This attosecond--femtosecond time scale separation has inspired the development of new numerical methods in the past few years~\cite{WangLiWang2015,MaWangWang2015,JiaAnWangEtAl2018,YostYaoKanai2019,AnLin2020}. In fact, when $\rho_0$ is a pure state, such a time scale separation is \textit{unphysical} and originates purely from a gauge choice of electron wavefunctions in the Schr\"odinger representation. In the simplest case, when the rank of the pure state $\rho_0$ is $1$, the gauge is only a time-dependent complex phase factor. When the rank of the pure state $\rho_0$ is $N$, the gauge is an $N\times N$ unitary matrix. The gauge choice does not affect $\rho$. In other words, a gauge transformed dynamics may solve a different set of variables from those in the Schr\"odinger dynamics, but results in the same physical observables.

Given a pure initial state, among all possible gauge choices, the parallel transport (PT) gauge~\cite{JiaAnWangEtAl2018,AnLin2020} yields the slowest gauge-transformed dynamics at any given time. Compared to the Schr\"odinger dynamics, the time step $h$ in the PT dynamics can be chosen to be much larger and is comparable to that of the von Neumann dynamics~\cref{eq:vonN}. When combined with implicit integrators (such as the Crank-Nicolson method or the implicit midpoint rule), the PT dynamics has been applied to rt-TDDFT simulations for real materials with thousands of atoms at the level of generalized gradient approximation exchange-correlation functionals (GGA, such as the Perdew--Burke--Ernzerhof~\cite{PerdewBurkeErnzerhof1996} functional)~\cite{JiaAnWangEtAl2018} and hybrid exchange-correlation functionals (such as the Heyd--Scuseria--Ernzerhof~\cite{HeydScuseriaErnzerhof2003} functional)~\cite{JiaLin2019,JiaWangLin2019}.

In previous works, the PT dynamics was derived for a pure initial state, and its efficiency has been justified in the linear, near adiabatic regime \cite{AnLin2020} in terms of a singularly perturbed linear system. The pure initial state is suitable for describing molecules and insulating materials at zero temperature. On the other hand, in practice, the initial state is often a low-energy excited state \cite{ElliottMaitra2012,FischerCramerGovind2015,BrunerHernandezMaugerEtAl2017}, or a thermal state \cite{YehManjanathChengEtAl2020} especially for metallic systems. This inspires us to consider the most general setting when $\rho_0$ is given by a mixed state (for instance, the occupation number of $\rho_0$ is given by the Fermi-Dirac distribution).

\vspace{1em}
\noindent\textbf{Contribution:}

By assuming a dynamical low-rank factorization $\rho(t)\approx \Phi(t)\sigma(t)\Phi^{\dag}(t)$, where $\Phi(t)\in \CC^{N_g\times N}$ and $\sigma(t)\in \CC^{N\times N}$, we derive the PT dynamics in terms of its low-rank factors $\Phi(t),\sigma(t)$. The PT dynamics with a pure initial state is recovered by setting $\sigma(t)=I_N$.
When the spectral radius of the Hamiltonian is large, the time step $h$ is simultaneously constrained by accuracy and stability requirements, and implicit integrators are more suited for efficient propagation of the PT dynamics. Using the implicit midpoint (IM) rule (also known as the second order Gauss-Legendre method, GL2) as an example, we derive the discretized numerical scheme, and prove that the resulting PT-IM scheme has certain orthogonality and trace-preserving properties. 

We then derive a new error bound for the discretized PT dynamics. Instead of relying on the linear quantum adiabatic theorem to obtain an \textit{a priori} error bound of the solution, our new error bound expresses the local truncation error directly in terms of the Hamiltonian, density matrix, and their derived quantities. Our analysis shows that an upper bound of the local truncation error of PT dynamics only involves certain commutators between the Hamiltonian (or its time derivatives) and the density matrix (or the associated spectral projector), while that of the Schr\"odinger gauge involves additional terms lacking such commutator structures. Using the commutator type error bound, in the near adiabatic regime when the \textit{a priori} estimate is available from the quantum adiabatic theorem, our new result shows the PT dynamics gains one extra order of accuracy in terms of the singularly perturbed parameter $\epsilon$ than the Schr\"odinger dynamics, which reproduces the previous result \cite{AnLin2020}. Recently, the quantum adiabatic theorem has been extended to certain weakly nonlinear systems~\cite{Fermanian-KammererJoye2020,GangGrech2017}. Our commutator type error analysis can be directly combined with such analysis leading to results comparable to that of~\cite{AnLin2020} in the weakly nonlinear regime. Away from the near adiabatic regime, the commutator scaling of the PT dynamics can still lead to a significantly smaller error than that of the Schr\"odinger dynamics. We illustrate the numerical performance of the PT dynamics for a number of one-dimensional model metallic systems, which also verifies the effectiveness of the new error bound.


\vspace{1em}
\noindent\textbf{Related works:}

Numerical integrators for rt-TDDFT simulation following the Schr\"odinger dynamics is a well-studied subject (see an early paper~\cite{CastroMarquesRubio2004}, and also~\cite{GomezMarquesRubioEtAl2018,RehnShenBuchholzEtAl2019} for recent comparative studies of a variety of standard numerical integrators), but the importance and the benefit of gauge-transformed dynamics have only been realized recently (see~\cite{YostYaoKanai2019} for another type of gauge-transformed dynamics using Wannier functions). 

At the continuous level, the PT dynamics is a special case of the dynamical low-rank approximation (DLRA) developed by Lubich \textit{et al.} (see \cite{KochLubich2007,Lubich2008book} for examples; DLRA is intimately related to the Dirac--Frenkel/McLachlan variational principle in the physics literature). The basic strategy of DLRA is to update a low-rank decomposition (such as eigenvalue or singular value decomposition) of a large matrix (in this case $\rho(t)$) on the fly. For a mixed initial state, a direct application of DLRA involves $\sigma^{-1}(t)$ in the equation of the low-rank factors, which in general can be a source of numerical instability~\cite{KochLubich2007,LubichOseledets2014}. Our derivation of the PT dynamics with a mixed initial state uses the structure of the von Neumann equation and can be viewed as a simplified derivation of DLRA. It also naturally shows that the pathological term $\sigma^{-1}(t)$ does not appear, so the PT dynamics is numerically stable even if one overestimates the numerical rank of $\rho(t)$. 

Regarding the time discretization, existing works of DLRA mostly use explicit integrators, although the possibility of using implicit integrators has also been mentioned in certain settings~\cite{LubichOseledets2014}. Our previous studies suggest that for rt-TDDFT calculations, the combined use of the PT dynamics and implicit integrators is the key for efficient propagation in real chemical and materials systems~\cite{JiaAnWangEtAl2018,AnLin2020}. 
The PT dynamics with a mixed initial state can also be viewed as a special case of the low-rank approximation for solving Lindblad equations by Le Bris \textit{et al.} \cite{LeBrisRouchon2013,LeBrisRouchonRoussel2015} (since the von Neumann equation can be viewed as the Lindblad equation without the decoherence operator), which is also derived independently of DLRA. It is worth pointing out that \cite{LeBrisRouchon2013} introduces an arbitrary Hermitian matrix that can be freely determined. We demonstrate that in the context of the von Neumann dynamics, setting this arbitrary matrix to $H(t)$ (the instantaneous Hamiltonian matrix), and $0$ (the zero matrix) leads to the Schr\"odinger dynamics and the PT dynamics, respectively.

\vspace{1em}
\noindent\textbf{Organization:}

The rest of the paper is organized as follows. In \cref{sec:prelim}, we introduce some preliminaries of rt-TDDFT and the PT dynamics with a pure initial state. We derive the PT dynamics with a mixed initial state in \cref{sec:pt_mixed}, and an implicit numerical propagator for the PT dynamics in \cref{sec:num_scheme}. \cref{sec:err_analysis} analyzes the numerical errors of the PT and the Schr\"odinger dynamics. Finally, we validate the error analysis with numerical results in \cref{sec:num_results}. For completeness, an alternative derivation of the PT dynamics that explicitly uses the structure of the tangent manifold (which is also a simplified derivation of~\cite{LeBrisRouchon2013}) is given in \cref{append:pt_tangent}. 
\cref{sec:anderson} describes Anderson's mixing method for solving the set of nonlinear equations in the PT-IM method.

\section{Preliminaries} \label{sec:prelim}

In this section, we briefly review the key idea of deriving PT dynamics with a pure initial state \cite{AnLin2020}. 
In the setting with a pure initial state, real-time time-dependent density functional theory (rt-TDDFT) solves the following set of Schr\"odinger equations
\begin{equation} 
  \I  \partial_{t} \Psi(t) = H(t,\rho(t)) \Psi(t), \quad \Psi(0)=\Psi_0.
  \label{eqn:tddft}
\end{equation}
Here $\Psi(t)=[\psi_{1}(t),\ldots,\psi_{N}(t)]$ is the collection of electron wavefunctions (also called electron orbitals), and the number of columns $N$ is equal to the number of electrons denoted by $N_e$ (spin degeneracy omitted).
The initial set of wavefunctions satisfy the orthonormality condition $\Psi(0)^{\dag}\Psi(0)=I_N$. Here $A^{\dag}$ denotes the Hermitian conjugate of a matrix or vector $A$.
The density matrix is $\rho(t) = \Psi(t)\Psi^{\dag}(t)\equiv \sum_{i=1}^N \psi_i(t)\psi_i^{\dag}(t)$, and in particular $\rho_0:=\rho(0)=\Psi(0)\Psi(0)^{\dag}$ is a pure state satisfying $\rho_0^2=\rho_0$. 

Throughout the paper we are concerned with time propagation instead of spatial discretization. Unless otherwise specified, \cref{eqn:tddft} represents a discrete, finite dimensional quantum system, i.e. $H(t, \rho)$ is a Hermitian matrix with finite dimension $N_g$. If the quantum system is spatially continuous, we may first find a set of orthonormal basis functions and expand the continuous wavefunction under this basis. Then after a Galerkin projection, \cref{eqn:tddft} becomes an $N_g$-dimensional quantum system, and $\psi_j(t)$ represents the coefficient vector under the basis for the $j$-th wavefunction. 


The time-dependent Hamiltonian operator $H(t,\rho(t))$ is Hermitian for all $t$ and $\rho$, and its precise form is not important for the purpose of this paper. Starting from a pure initial state $\rho_0$, the orthogonality condition $\Psi(t)^{\dag}\Psi(t)=I_N$ is satisfied for all $t\ge 0$, and hence $\rho(t)$ is a pure state for all $t$ satisfying $\rho^2(t)=\rho(t)$. Throughout the paper, we may use the notations $\partial_t\rho=\rho_t=\dot{\rho}$ interchangeably for the time-derivatives. For composite functions such as $H(t,\rho(t))$, we use the notation $\dot{H} : = \frac{d}{dt} H(t,\rho(t)) = H_t +  H_\rho \rho_t$, where the tensor contractions are defined such that the chain rule holds. For example, the tensor contraction between the 4-tensor $H_\rho$ and the matrix $\rho_t$ are defined such that the chain rule $\frac{d}{dt} H(t, \rho(t)) = H_t + H_\rho \rho_t$ follows the element-wise operation
$$\frac{d}{dt} H_{ij}(t,\rho(t))  = \partial_t H_{ij}(t,\rho(t)) 
+ \sum_{k,l} \frac{\partial H_{ij}}{\partial \rho_{kl}}(t,\rho(t))  \frac{\partial \rho_{kl}(t)}{\partial t}.
$$

The set of Schr\"odinger equations \eqref{eqn:tddft} is equivalent to the von Neumann dynamics \eqref{eq:vonN}. 
Note that if we right multiply $\Psi(t)$ by a time-dependent unitary matrix $U(t)\in \CC^{N\times N}$ and let $\Phi(t)=\Psi(t)U(t)$, then
\begin{equation}
  \rho(t)= \Psi(t)\Psi^{\dag}(t) =  \Phi(t)\left[U^{\dag}(t)U(t)\right]\Phi^{\dag}(t)
  = \Phi(t)\Phi^{\dag}(t).
  \label{eqn:dmgauge}
\end{equation}
The unitary rotation matrix $U(t)$ is called the gauge matrix, and \cref{eqn:dmgauge} indicates that the density matrix is \textit{gauge-invariant}. In particular, the choice $U(t)=I_N$ is referred to as the Schr\"odinger gauge.

Since all physical observables can be derived from the von Neumann equation \eqref{eq:vonN} and the density matrix $\rho(t)$, the choice of the gauge matrix $U(t)$ has no measurable effects. 
On the other hand, the gauge matrix introduces additional degrees of freedom, and can oscillate at a different time scale from that of the corresponding wavefunctions. 
It is then desirable to optimize the gauge matrix, so that 
the transformed wavefunctions $\Phi(t)$ vary \emph{as slowly as
possible}, without changing the density matrix. This results in the following variational problem
\begin{equation}\label{eqn:minproblem}
  \min_{U(t)} \quad \norm{\dot{\Phi}}^2_{F},
  \ \text{s.t.} \ \Phi(t) = \Psi(t)U(t), U^{\dag}(t)U(t)=I_{N}.
\end{equation}
Here $\norm{\dot{\Phi}}^2_{F}:=\Tr[\dot{\Phi}^{\dag}\dot{\Phi}]$
measures the Frobenius norm of the time derivative of the transformed
orbitals. 

The minimizer of~\eqref{eqn:minproblem}, in terms of $\Phi$, 
satisfies the following equation
\begin{equation}
  \rho\dot{\Phi}=0.
  \label{eqn:PTcondition}
\end{equation}
We refer readers to \cite{AnLin2020} for the derivation. 
\cref{eqn:PTcondition} has an intuitive explanation that the optimal dynamics should minimize the ``internal'' rotations within the range of $\rho$.
\cref{eqn:PTcondition} implicitly defines a gauge
choice for each $U(t)$, and this gauge is called the \emph{parallel transport
gauge}. The name ``parallel transport'' comes from that 
$\Phi(t)$ can be identified as the unique horizontal lift~\cite{Nakahara2003} of $\rho(t)$ from the Grassmann manifold to the Stiefel manifold, starting from the initial condition $\Psi_0$. This will be further explained in \cref{sec:err_analysis}.

From \cref{eqn:PTcondition}, the governing equation of $\Phi(t)$ can be concisely written down as
\begin{equation}
  \I \partial_t \Phi(t) = H(t,\rho(t))\Phi(t) -
  \Phi(t)(\Phi^{\dag}(t)H(t,\rho(t))\Phi(t)), \quad \Phi(0)=\Psi_0,
  \label{eqn:pt}
\end{equation}
where $\rho(t) = \Phi(t)\Phi^{\dag}(t)$. 
Notice that \cref{eqn:pt} introduces one extra term compared to the original dynamics \cref{eqn:tddft} under Schr\"odinger gauge, and directly provides a self-contained definition of the transformed wavefunctions under the optimal gauge. 
In practice, we can directly solve \cref{eqn:pt} by numerical schemes to approximate the dynamics, instead of computing the PT gauge explicitly. 

To observe the advantage of the parallel transport dynamics,
consider the extreme case that each column of $\Psi_0$ is already an eigenstate of $H(0)$ and $H(t,\rho(t))\equiv H(0)$ is a time-independent matrix. Then \cref{eqn:pt} is reduced to
\[
\I \partial_t \Phi(t) =0.
\]
Hence $\Phi(t)=\Phi(0)$ holds for all $t\ge 0$, while each column of the solution Schr\"odinger dynamics \eqref{eqn:tddft} rotates with a time-dependent phase factor. For less trivial dynamics, the temporal oscillation of gauge-transformed wavefunctions $\Phi(t)$ can still be significantly slower than that of $\Psi(t)$.

\section{Parallel transport dynamics with a mixed initial state} \label{sec:pt_mixed}

In rt-TDDFT calculations, the pure initial state can be used for simulating insulating systems starting from the ground state, or a well-defined excited state. In many other cases the initial state should be a mixed state. For instance, for metallic systems at finite temperature, the initial state often takes the form of the Fermi-Dirac distribution
\begin{equation} \label{eq:fermi_dirac_initial}
\rho(0) = \left(1 + \exp{(\beta(H(0) -\mu) )}\right)^{-1},
\end{equation}
where $\beta = 1/(k_\text{B}T)$, $k_\text{B}$ is the Boltzmann constant, $T$ is the temperature. The chemical potential $\mu$ is a Lagrange multiplier, which should be adjusted to satisfy the normalization condition 
\begin{equation}
\Tr[\rho(0)]=N_e, 
\label{eqn:rho0_normalize}
\end{equation}
where $N_e$ is the number of electrons.
If we diagonalize $H(0)$ according to $H(0)\psi_i(0)=\veps_i(0)\psi_i(0)$, then the occupation number 
\[
s_i(0):=\braket{\psi_i(0)|\rho(0)|\psi_i(0)}=\left(1 + \exp{(\beta(\veps_i(0) -\mu) )}\right)^{-1}.
\]
Hence when $\beta$ is large (e.g. at room temperature $300$K, $\beta\approx 10^3$ in the atomic unit), $s_i(0)$ is very close to $0$ when $\veps_i(0)-\mu \gg \beta^{-1}$. Therefore, $\rho(0)$ can be very well approximated by a low rank matrix, with its approximate rank denoted by $N$. In other words, we can set
\[
\rho(0)=\sum_{i=1}^{N} \psi_i(0)s_i(0)\psi_i^{\dag}(0)=\Psi(0)\sigma(0)\Psi^{\dag}(0),
\]
with the chemical potential $\mu$ slightly adjusted so that the normalization condition \eqref{eqn:rho0_normalize} is still satisfied. Here $\sigma_0:=\sigma(0)=\diag[s_1(0),\ldots,s_N(0)]$ is a diagonal matrix. Since the occupation number satisfies $0< s_i(0)< 1$, we have $N> N_e$, and $\rho^2(0)\prec \rho(0)$. We also assume $N_g\gg N$.

If we solve the nonlinear Schr\"odinger equation \eqref{eqn:tddft} to obtain $\Psi(t)$, then
\begin{equation} \label{eqn:rho_def_schd}
\rho(t)=\Psi(t)\sigma_0\Psi^{\dag}(t),
\end{equation}
is the unique solution of \cref{eq:vonN} (viewed as a large ODE system) with the initial state $\rho(0)$. Hence in practice, we only need to solve \cref{eqn:tddft} in the same way as for the pure initial state, but weigh the contribution of each time-dependent vector $\psi_i(t)$ always by the \textit{initial} occupation number $\sigma_i(0)$. This fact that the occupation number $\sigma(t)$ remains as a constant matrix $\sigma_0$ can also be derived directly (see \cref{eq:sch_main} in \cref{append:pt_tangent}). 
However, similar to the case with a pure initial state, \cref{eqn:tddft} can require a relatively small time step size. 

Note that we may still apply a gauge matrix $U(t)\in\CC^{N\times N}$ and define $\Phi(t)=\Psi(t)U(t)$ with initial condition $U(0)=I_N$. In such a case, we must also redefine the occupation number matrix as 
\begin{equation}
\sigma(t)=U^{\dag}(t)\sigma_0 U(t),
\label{eqn:sigma_t}
\end{equation}
so that
\begin{equation}
\rho(t)=\Phi(t)\sigma(t)\Phi^{\dag}(t)
\label{eqn:rhot_rotate_mixed}
\end{equation}
is satisfied. Here $\sigma(t)$ is now a Hermitian matrix of size $N$ and may no longer be diagonal for $t > 0$. 
We would like to solve again the optimization problem in \cref{eqn:minproblem} so that the gauge-transformed wavefunctions $\Phi(t)$ vary as slowly as possible. This leads to \cref{eqn:PTcondition} and hence \cref{eqn:pt}, with $\rho(t)$ defined in \cref{eqn:rhot_rotate_mixed}. For simplicity, we may also define a gauge-invariant projector
\[
P(t)=\Psi(t)\Psi^{\dag}(t)=\Phi(t)\Phi^{\dag}(t),
\]
so that \cref{eqn:pt} can be rewritten as
\[
\I \partial_t \Phi(t) = (I-P(t))H(t,\rho(t))\Phi(t).
\]
Here the identity matrix is given as $I=I_{N_g}$, and we have used that $P(t)\Phi(t)=\Phi(t)$. 

In order to close the equation, it remains to identify the equation of motion of $\sigma(t)$. First, by differentiating the equation $ \Psi(t) U(t) = \Phi(t) $
and using \eqref{eqn:pt}, we may derive the dynamics of the gauge $U(t)$, i.e.
\[
(\I\partial_t \Psi(t)) U(t)+\Psi(t)(\I\partial_t U(t))=H(t,\rho(t))\Psi(t)U(t)-\Psi(t)\Psi^{\dag}(t)H(t,\rho(t))\Psi(t)U(t).
\]
This gives
\begin{equation}
\I \partial_t U(t)=-(\Psi^{\dag}(t)H(t,\rho(t))\Psi(t))U(t).
\label{eqn:gauge_dynamics}
\end{equation}
By differentiating both sides of \cref{eqn:sigma_t} and using \cref{eqn:gauge_dynamics}, we have
\[
\begin{split}
\I \partial_t \sigma(t)&=(\I \partial_t U^{\dag}(t))\sigma_0 U(t)+U^{\dag}(t)\sigma_0 (\I \partial_t U(t))\\
&= U^{\dag}(t) (\Psi^{\dag}(t)H(t,\rho(t))\Psi(t))\sigma_0 U(t)-U^{\dag}(t)\sigma_0 (\Psi^{\dag}(t)H(t,\rho(t))\Psi(t))U(t)\\
&= \Phi^{\dag}(t)H(t,\rho(t))\Phi(t)\sigma(t)-\sigma(t)\Phi^{\dag}(t)H(t,\rho(t))\Phi(t)\\
&= [\Phi^{\dag}(t)H(t,\rho(t))\Phi(t),\sigma(t)].
\end{split}
\]
Therefore the equation of motion for $\sigma(t)$ only depends on the slowly varying gauge-transformed wavefunctions $\Phi(t)$.

In summary, the parallel transport dynamics with a general initial state consists of the following set of equations
\begin{equation}
\begin{split}
  \I \partial_t \Phi(t) &= (I-P(t))H(t,\rho(t))\Phi(t), \\
  \I \partial_t \sigma(t)&=[\Phi^{\dag}(t)H(t, \rho(t))\Phi(t),\sigma(t)],\\
  \rho(t)&=\Phi(t)\sigma(t)\Phi^{\dag}(t), \quad P(t)=\Phi(t)\Phi^{\dag}(t),\\
  \Phi(0)&=\Psi_0, \quad \sigma(0)=\sigma_0.
\end{split}
\label{eqn:pt_mix}
\end{equation} 

\cref{eqn:pt_mix} gives a self-contained definition of the transformed wavefunctions $\Phi(t)$ and the matrix $\sigma(t)$ under the optimal gauge. 
Therefore, we can directly solve \cref{eqn:pt_mix} to numerically approximate the state $\rho(t)$ without computing the PT gauge matrix explicitly. 
Notice that, compared to the Schr\"odinger dynamics in which one can set $\sigma(t) = \sigma(0)$, the PT dynamics in \cref{eqn:pt_mix} introduces one extra nonlinear term in the propagation of the wavefunctions, and enlarges the size of the ODE system via a non-trivial dynamics of the transformed $\sigma(t)$. 
This is different from the pure state setting where only an extra term in the equation of $\Phi$ is added. 
However, due to the assumption that $N_g \gg N$, the increase of the number of variables by $N^2$ due to $\sigma(t)$ does not add too much overhead in the numerical simulation. 

In \cref{append:pt_tangent}, we provide an alternative derivation of \cref{eqn:pt_mix} using the tangent space formulation, which follows the derivation in \cite{LeBrisRouchon2013,LeBrisRouchonRoussel2015} for Lindblad equations, and is more analogous to the derivation of the dynamical low-rank approximation. The derivation also provides an alternative perspective of the gauge choice in terms of an auxiliary Hamiltonian.

\section{Numerical propagation of the parallel transport dynamics}
\label{sec:num_scheme}

In order to solve \cref{eqn:pt_mix} numerically, for simplicity we assume that a uniform time discretization $t_{n}=n h$, and $h$ is the
time step size. The numerical values of $\Phi(t),\sigma(t),\rho(t),P(t)$ at time $t=t_n$ are denoted by $\Phi_n,\sigma_n,\rho_n,P_n$, respectively, and we define $H_{n}=H(t_{n},\rho_{n})$. Previous studies suggested that when the spectral radius of $H$ is large, the PT dynamics should be solved using implicit time integrators~\cite{JiaAnWangEtAl2018,AnLin2020}. 
This allows one to use a time step much larger than $\norm{H}^{-1}$, and the result from the PT dynamics can be much more accurate than that from the Schr\"odinger dynamics using the same step size.

In order to discretize the PT dynamics with a mixed initial state, we consider the implicit midpoint (IM) rule (also known as the Gauss-Legendre method of order 2). We introduce the shorthand notations
\begin{equation}
\Phi_{n+\frac12}=\frac12(\Phi_n+\Phi_{n+1}), \quad \sigma_{n+\frac12}=\frac12(\sigma_n+\sigma_{n+1}),
\label{eqn:PT_GL2_midpoint}
\end{equation}
and accordingly
\begin{equation}
P_{n+\frac12} = \Phi_{n+\frac12} (\Phi^{\dag}_{n+\frac12}\Phi_{n+\frac12})^{-1}\Phi_{n+\frac12}^\dagger, \quad 
\rho_{n+\frac12} = \Phi_{n+\frac12}\sigma_{n+\frac12} \Phi_{n+\frac12}^\dagger, \quad 
H_{n+\frac12} = H\left(t_{n+\frac12}, \rho_{n+\frac12}\right).
\label{eqn:half_step}
\end{equation}
We remark that $\rho_{n+\frac12}$ is only a shorthand notation and may not be an admissible density matrix. In particular, even if 
$\Tr[\rho_n]=\Tr[\rho_{n+1}]=N_e$ (see \cref{prop:ortho}), we may not have $\Tr[\rho_{n+\frac12}]=N_e$.  
On the other hand, $P_{n+\frac12}$ is still a projector satisfying $P_{n+\frac12}\Phi_{n+\frac12}=\Phi_{n+\frac12}$. 

With these notations, the parallel transport-implicit midpoint scheme (PT-IM) reads
\begin{align}
\I \frac{\Phi_{n+1} - \Phi_n}{h} &=  (I - P_{n+\frac12})H_{n+\frac12} \Phi_{n+\frac12},\label{eqn:PT_GL2_Phi}\\
\I \frac{\sigma_{n+1} -\sigma_n}{h} & =\left[\Phi_{n+\frac12}^\dagger H_{n+\frac12} \Phi_{n+\frac12}, \sigma_{n+\frac12}\right] \label{eqn:PT_GL2_Sigma},
\end{align}
which form a set of nonlinear algebraic equations and need to be solved self-consistently. We can rewrite \cref{eqn:PT_GL2_Phi,eqn:PT_GL2_Sigma} as 
\begin{equation}
\begin{split}
\Phi_{n+1}&=\Phi_n + \frac{h}{\I } (I - P_{n+\frac12})H_{n+\frac12} \Phi_{n+\frac12},\\
\sigma_{n+1} &=\sigma_n + \frac{h}{\I }\left[\Phi_{n+\frac12}^\dagger H_{n+\frac12} \Phi_{n+\frac12}, \sigma_{n+\frac12}\right].
\end{split}
\label{eqn:one_step}
\end{equation}
If we choose $(\Phi_{n+1},\sigma_{n+1})$ to be the unknowns and identify it with a vector $x\in\CC^{N_g N+N^2}$, then \cref{eqn:PT_GL2_Phi,eqn:PT_GL2_Sigma} can be viewed as a fixed point equation in the abstract form 
\[
x=T(x).
\]
The structure of this fixed point problem resembles that of the self-consistent field iterations (SCF) in 
standard electronic structure calculations~\cite{Martin2008}. 
Here we use Anderson's mixing method \cite{Anderson1965} to solve this fixed problem (see \cref{sec:anderson}).

The following proposition shows that PT-IM preserves the orthogonality as well as the trace condition. 

\begin{prop}\label{prop:ortho}
Assume $\Phi_n^{\dag}\Phi_n=I_N$, $\sigma_n=\sigma_n^{\dag}$, $\Tr[\sigma_n]=N_e$, and that \cref{eqn:PT_GL2_Phi,eqn:PT_GL2_Sigma} have a unique solution $(\Phi_{n+1},\sigma_{n+1})$, then the solution satisfies
\begin{equation}
\Phi_{n+1}^{\dag}\Phi_{n+1}=I_N, 
\label{eqn:phi_conserve}
\end{equation}
and
\begin{equation}
\sigma^{\dag}_{n+1}=\sigma_{n+1},\quad \Tr[\sigma_{n+1}]=N_e, \quad \Tr[\sigma_{n+1}^2]=\Tr[\sigma_n^2]. \label{eqn:sigma_conserve}
\end{equation}
As a consequence, we have $\Tr[\rho_{n+1}]=\Tr[\rho_n]=N_e$. 
\end{prop}
\begin{proof}
First, use the definition in \eqref{eqn:PT_GL2_midpoint} and apply $\Phi_{n+\frac12}^\dagger$ to both sides of \eqref{eqn:PT_GL2_Phi}, we obtain
\[
\frac{\I}{2h}(\Phi_{n+1}^{\dag} \Phi_{n+1}-\Phi_{n}^{\dag} \Phi_{n}) - \frac{\I}{2h}(\Phi_{n+1}^{\dag} \Phi_{n}-\Phi_{n}^{\dag} \Phi_{n+1})=\Phi^{\dag}_{n+\frac12}(I - P_{n+\frac12})H_{n+\frac12}\Phi_{n+\frac12}=0.
\]
On the left-hand side, the first term is anti-Hermitian and the second term is Hermitian. So both terms must vanish, and 
\[
\Phi_{n+1}^{\dag}\Phi_{n+1}=\Phi_{n}^{\dag}\Phi_{n}=I_N.
\]
This proves \cref{eqn:phi_conserve}.

Second, denote by $\wt{H}:=\Phi_{n+\frac12}^\dagger H_{n+\frac12} \Phi_{n+\frac12}$, we may solve the equation
\[
\sigma_{n+1}-\sigma_n=-\frac{\I h}{2}[\wt{H},\sigma_n]-\frac{\I h}{2}[\wt{H},\sigma_{n+1}]
\]
to obtain $\sigma_{n+1}$. Applying the Hermite conjugation to both sides and using that $\wt{H},\sigma_n$ are Hermitian matrices, we have
\[
\sigma^{\dag}_{n+1}-\sigma_n=-\frac{\I h}{2}[\wt{H},\sigma_n]-\frac{\I h}{2}[\wt{H},\sigma^{\dag}_{n+1}].
\]
The uniqueness of $\sigma_{n+1}$ implies $\sigma_{n+1}=\sigma^{\dag}_{n+1}$. Moreover, since the right-hand side of \cref{eqn:PT_GL2_Sigma} is traceless, we have $\Tr[\sigma_{n+1}]=\Tr[\sigma_n]$.

Finally, applying $\sigma_{n+\frac12}$ from the left to both sides of \cref{eqn:PT_GL2_Sigma}, we have
\[
\frac{\I}{2h} (\sigma_{n+1}^2-\sigma_n^2-\sigma_{n+1}\sigma_n+\sigma_n \sigma_{n+1})=\sigma_{n+\frac12}[\wt{H}, \sigma_{n+\frac12}].
\]
Since the right-hand side is traceless, by taking trace of both sides we obtain 
\[
\Tr[\sigma_{n+1}^2]=\Tr[\sigma_{n}^2].
\]
This finishes the proof of the equalities in \eqref{eqn:sigma_conserve}. 
\end{proof}

\cref{eqn:phi_conserve} can be viewed as a consequence of the general fact that the PT-IM method preserves quadratic invariants, and in particular orthogonality constraints (see e.g. \cite[pp 132]{HairerLubichWanner2006} for a more general description of orthogonality preserving Runge-Kutta methods). \cref{prop:ortho} confirms that the PT-IM scheme preserves orthogonality of $\Phi(t)$, as well as the number of electrons.

\section{Error analysis}  \label{sec:err_analysis}
In this section, we consider the numerical error of the PT-IM scheme for concreteness, and compare the form of the error terms with those from the Schr\"odinger dynamics. The error analysis can also be extended to other Runge-Kutta methods and linear multistep methods.

\subsection{Error analysis of the PT dynamics} \label{sec:err_analysis_pt}

Before proceeding with the detailed error analysis, we first provide some abstract perspectives. Let the local truncation error be defined as $e_k(X) = X(t_k) - \widetilde{X}_k$, 
where $t_k = kh$ and $\widetilde{X}_k$ represents the numerical solution of $X$ at the 
$k$-th step with previous step to be exactly $X(t_{k-1})$, where $X$ is the concatenation of $\Phi$ and $\sigma$, namely $\begin{pmatrix} \Phi \\ \sigma\end{pmatrix}$. Since IM is a second order method, the local truncation error can be bounded in terms of the third order derivatives \cite{HairerNorsettWanner1987}
\begin{equation} \label{eqn:lte_ineq}
    \norm{e_k(X)} \le C\max_{t\in[t_{k-1},t_k]}\norm{\partial_t^3X(t)} h^3.
\end{equation}
Here $C$ is an absolute constant depending only on the choice of the numerical scheme.

 Note that $P$ is a rank-$N$ projector, and can be identified with the Grassmann manifold $\mathrm{Gr}(N_g,N;\CC)$, i.e. the $N$-dimensional subspace of $\CC^{N_g}$. On the other hand, the gauge-transformed wavefunctions $\Phi$ belongs to the Stiefel manifold $\mathrm{St}(N_g,N;\CC)$, which is the set of first $N$ columns of an $N_g$-dimensional unitary matrices. The Grassmann manifold is the quotient space of $\mathrm{St}(N_g,N;\CC)$ by $\mathrm{U}(N)$, denoted by
\[
\mathrm{Gr}(N_g,N;\CC)=\mathrm{St}(N_g,N;\CC)/\mathrm{U}(N)
\]
The projector $P(t)$ can be identified with a curve in $\mathrm{Gr}(N_g,N;\CC)$, obtained by solving the von Neumann equation. On the other hand, the wavefunctions $\Psi(t),\Phi(t)$ in the Schr\"odinger and the parallel transport gauge are \textit{lifts} of the curve $P(t)$ from the quotient space to $\mathrm{St}(N_g,N;\CC)$. In particular, $\Phi(t)$ can be identified as the unique \textit{horizontal lift}~\cite{Nakahara2003} of $P(t)$, starting from the initial condition $\Psi_0$ (which fixes a gauge choice initially). We have demonstrated that the parallel transport gauge yields the slowest dynamics in the sense of minimizing $\norm{\partial_t \Phi}_F$.
For simplicity, in the following discussions we will consider the operator norm $\norm{\cdot}$ for $\Phi$, $P$ and their time derivative. We expect that the size of the $k$-th order time derivative $\norm{\partial_t^k \Phi}$ should also be bounded by that of $\norm{\partial_t^k P}$. On the other hand, $\norm{\partial_t^k\Psi}$ may not be bounded by $\norm{\partial_t^k P}$ due to the gauge matrix.

Recall the relation
\[
P\Phi=\Phi, \quad P\partial_t \Phi=0,
\]
and this gives
\[
\partial_t\Phi=(\partial_t P) \Phi.
\]
Keep differentiating and obtain
\[
\partial^2_t \Phi =[\partial^2_t P+(\partial_t P)^2]\Phi, 
\quad
\partial^3_t \Phi = [\partial_t^3 P + 2 (\partial_t^2 P) (\partial_t P) + (\partial_t P) (\partial_t^2 P) +(\partial_t P)^3] \Phi.
\]
Using the fact that $\norm{\Phi}=1$, 
we have
\[
\norm{\partial_t\Phi}\le \norm{\partial_t P}.
\]
Similarly
\[
\norm{\partial^2_t\Phi}\le \norm{\partial^2_t P+(\partial_t P)^2}\le \norm{\partial^2_t P}+\norm{(\partial_t P)}^2, 
\]
and
\[
\norm{\partial^3_t\Phi}\le \norm{\partial^3_t P}+3\norm{\partial^2_t P} \norm{\partial_t P}+\norm{\partial_t P}^3.
\]
This implies that $\norm{\partial_t^{k} \Phi}$ is controlled by  $\norm{\partial_t^{\ell} P}$ with $\ell\le k$. 
On the other hand, 
\[
\sigma = \Phi^\dagger \rho \Phi
\]
implies that the time derivative $\norm{\partial_t^k \sigma(t)}$ is controlled by $\norm{\partial_t^{\ell} \Phi}$ and $\norm{\partial_t^{\ell} \rho}$ with $\ell \le k$. To be specific, a direct computation gives
\begin{align*}
        \partial_t^3 \sigma =& (\partial_t^3 \Phi^\dagger )\rho \Phi + \Phi^\dagger (\partial_t^3 \rho) \Phi + \Phi^\dagger \rho (\partial_t^3 \Phi) + 6 (\partial_t \Phi^\dagger) (\partial_t \rho) (\partial_t \Phi)\\
    &+ 3 (\partial_t^2 \Phi^\dagger) (\partial_t \rho )\Phi + 3 (\partial_t^2 \Phi^\dagger )\rho (\partial_t \Phi) + 3 (\partial_t \Phi^\dagger) (\partial_t^2 \rho) \Phi + 3 (\partial_t \Phi^\dagger) \rho (\partial_t^2 \Phi) + 3 \Phi^\dagger (\partial_t^2 \rho) (\partial_t \Phi) + 3 \Phi^\dagger (\partial_t \rho) (\partial_t^2 \Phi),
\end{align*}
and hence
\begin{align*}
        \norm{\partial_t^3 \sigma} \le& 2\norm{ \partial_t^3 \Phi} 
        + \norm{\partial_t^3 \rho}  
        + 6 \norm{\partial_t \Phi^\dagger} \norm{\partial_t\rho} \norm{\partial_t \Phi}\\
    &
    + 6 \norm{\partial_t^2 \Phi} \norm{\partial_t \rho} 
    + 6 \norm{\partial_t^2 \Phi} \norm{\partial_t \Phi}  
    + 6 \norm{\partial_t \Phi}\norm{\partial_t^2 \rho},
\end{align*}
where we used the facts that $\norm{\Phi}=1$ and $\norm{\rho}\le 1$.

To summarize, the error analysis of the PT dynamics is reduced to the estimate of $\norm{\partial^{k}_t P}$ and $\norm{\partial_t^{k} \rho}$. In particular, for the analysis of PT-IM, we need $k\le 3$.

\begin{lem}\label{lem:der_P}
Suppose $H(t,\rho)$ is continuously differentiable in terms of $t$ and $\rho$ up to second order. Then the derivatives of $P$ satisfy
\begin{align} \label{eq: estimate_der_P}
\norm{\partial_t P}  \le &  \norm{[H,P]}, \\
\norm{\partial_t^2 P}  \le & \norm{[H_t, P]} + \norm{H_\rho[H,\rho]} + \norm{[H, [H,P]]} , 
\label{eq: estimate_der_P_eq2}\\
\norm{\partial_t^3 P}  \le & \norm{ [H_{tt},P] } + 2 \norm{(H_{t})_\rho [H,\rho]}  + \norm{H_{\rho\rho} ([H,\rho])^2 }  
\nonumber 
 + \norm{ H_\rho [H_t,\rho]  } \\ & + \norm{ H_\rho [H_\rho [H,\rho],\rho]}  + \norm{ H_\rho  [H,[H, \rho]]} \nonumber
 + 2\norm{ [H_t, [H,P]]} + 2\norm{ [H_\rho [H,\rho], [H,P]] }  \\&
+ \norm{[H, [H_t,P]]} + \norm{ [H, [H_\rho[H,\rho] ,P]]} + \norm{ [H, [H,[H,P]]]},
\end{align}
where the subscripts denote the partial derivatives.
\end{lem}
\begin{proof} 
The first inequality is trivial.
To prepare for the differentiation of $P$, we start by computing the derivatives of $H$. For notational simplicity, we use the subscripts to denote the partial derivative and omit the explicit $(t,\rho(t))$ dependence in $H$. The first order derivative of $H$ reads 
\begin{equation}\label{eq:dotH}
\dot H : = \frac{d}{dt} H(t,\rho(t)) = H_t +  H_\rho \rho_t= H_t - \I   H_\rho [H,\rho], 
\end{equation}
and the second order derivative of $H$ is given by
\[
\begin{split}
    \ddot {H} : =&  \frac{d^2}{dt^2} H(t, \rho(t)) =\frac{d}{dt} H_t -\I \frac{d}{dt}( H_\rho [H,\rho])\\
    =& H_{tt} +   (H_{t})_\rho \rho_t 
     -\I (H_t)_\rho  [H,\rho] 
     -\I H_{\rho\rho} \rho_t [H,\rho]
     -\I H_\rho [\dot H,\rho] 
     -\I H_\rho [H,\rho_t].    
\end{split}
\]
It follows from $\I  \partial_t \rho = [H, \rho]$ that  
\begin{equation}\label{eq:ddotH}
\begin{split}
\ddot H = & H_{tt} - 2\I (H_{t})_\rho [H,\rho] 
     - H_{\rho\rho} ([H,\rho])^2 
     \\
     & - \I H_\rho [H_t,\rho] 
     -  H_\rho [H_\rho [H,\rho],\rho] 
     - H_\rho  [H,[H, \rho]]. 
\end{split}
\end{equation}
The second order derivative of $P$ becomes
\begin{equation}\label{eq:d^2P}
\begin{split}
\partial_t^2{P} &= -\I \frac{d}{dt}([H(t, \rho(t)),P(t)]) 
        = -\I [\dot{H},P] -\I [H,\partial_t{P}] \\
        &= -\I [H_t,P] - [H_\rho[H,\rho] ,P] - [H,[H,P]],
\end{split}
\end{equation}
together with the fact that $\|P\| \le 1$, we obtain  \eqref{eq: estimate_der_P_eq2}.
Similarly, the third order derivative of $P$ can be computed explicitly via
\begin{equation*}
\begin{split}
\partial_t^3{P} 
    &= - \I [\ddot{H},P] - 2\I [\dot{H}, \partial_t P] - \I [H,\partial_t^2{P}] .
\end{split}
\end{equation*}
Plugging in \eqref{eq:dotH}, \eqref{eq:ddotH} and \eqref{eq:d^2P}, one obtains
\begin{equation*}
\begin{split}
\partial^3_t P = & 
- \I [H_{tt},P]  - 2 [(H_{t})_\rho [H,\rho],P]  + \I [H_{\rho\rho} ([H,\rho])^2 ,P] 
 -  [H_\rho [H_t,\rho] ,P] \\& + \I [H_\rho [H_\rho [H,\rho],\rho],P]  + \I [H_\rho  [H,[H, \rho]],P]
 - 2 [H_t, [H,P]] \\&+ 2\I [H_\rho [H,\rho], [H,P]]  
- [H, [H_t,P]] + \I [H, [H_\rho[H,\rho] ,P]] + \I [H, [H,[H,P]]].
\end{split}
\end{equation*}
Taking the norm yields the desired result.
\end{proof}

\begin{lem}\label{lem:der_rho}
Suppose $H(t,\rho)$ is continuously differentiable in terms of $t$ and $\rho$ up to second order. The derivatives of $\rho$ satisfy
\begin{align} \label{eq: estimate_der_rho}
\norm{\partial_t \rho}  \le &  \norm{[H,\rho]}, \\
\norm{\partial_t^2 \rho}  \le & \norm{[H_t, \rho]} + \norm{[H_\rho[H,\rho], \rho]} + \norm{[H, [H,\rho]]} , 
\label{eq: estimate_der_rho_eq2}\\
\norm{\partial_t^3 \rho}  \le & \norm{ [H_{tt},\rho] } + 2 \norm{[(H_{t})_\rho [H,\rho], \rho]}  + \norm{[H_{\rho\rho} ([H,\rho])^2, \rho] } 
\nonumber 
 + \norm{ [H_\rho [H_t,\rho], \rho]  } \\ & + \norm{ [H_\rho [H_\rho [H,\rho],\rho], \rho]}  + \norm{ [H_\rho  [H,[H, \rho]], \rho]} \nonumber
 + 2\norm{ [H_t, [H,\rho]]}  + \norm{[H, [H_t,\rho]]} 
 \\ &+ 2\norm{ [H_\rho [H,\rho], [H,\rho]] } 
 + \norm{ [H, [H_\rho[H,\rho] ,\rho]]} + \norm{ [H, [H,[H,\rho]]]},
\end{align}
where the subscripts denote the partial derivatives.
\end{lem}
\begin{proof} 
The proof is similar as \cref{lem:der_P} since $\rho$ satisfies the equation $\I \partial_t \rho = [H,\rho]$, which has the same form of that for $P$. 
\end{proof}
Therefore, the local truncation errors of the PT dynamics can be bounded by terms involving commutators of $[H, P]$, $[H, \rho]$, $[H_t, \rho]$, $[H_{tt}, \rho]$, $[H_t, P]$, $[H_{tt}, P]$.

\subsection{Comparison to the Schr\"odinger dynamics} \label{sec:err_analysis_schd}

In this section, we discuss the local truncation error of the Schr\"odinger dynamics and the global errors of the PT and Schr\"odinger dynamics. The local truncation error can be summarized in the following lemma. Note that in the bound, we keep the wavefunction $\Psi$ for the terms without commutator structures, such as $\norm{H^3\Psi}$, instead of replacing it by the operator norm $\norm{H^3}$, because the latter could be significantly larger than the former.

\begin{lem} \label{lem:LTE_2nd_schd}
For the IM scheme, 
the local truncation errors of Schr\"odinger dynamics \eqref{eqn:tddft} can be bounded as  
\begin{equation} \label{eqn:err_schd}
\begin{split}
\norm{e_k(\Psi)} \le & C\big( \norm{H^3 \Psi} +  \norm{H H_t \Psi} + 2\norm{H_t H \Psi} + \norm{H_{tt} \Psi}  \\
& + \norm{H H_\rho [H,\rho]} + 2\norm{H_\rho [H,\rho] H} + 2\norm{(H_{t})_\rho [H,\rho]}
+ \norm{H_{\rho\rho} ([H,\rho])^2} \\& + \norm{H_\rho [H_t,\rho]}  
+ \norm{H_\rho [H_\rho [H,\rho],\rho]} + \norm{H_\rho  [H,[H, \rho]]}
  \big),
\end{split}
\end{equation}
for some constant $C$ that does not depend on $t_k$, $h$.
\end{lem}
\begin{proof}
It suffices to calculate the derivatives of $\Psi$. 
The second order derivative is computed as
\begin{equation*}
\partial_t^2 \Psi =  -\I H \partial_t \Psi -\I \dot H \Psi = - H^2 \Psi -\I \dot H \Psi 
\end{equation*}
and the third order derivative can be computed as 
\begin{align*}
\partial^3_t \Psi  = &  -\I  H \ddot \Psi -2\I  \dot H \dot \Psi -\I \ddot H \Psi \\ 
                        = & \I H^3 \Psi   -  H \dot H \Psi 
                            - 2  \dot H H \Psi -\I \ddot H \Psi \\
                        = & \I H^3 \Psi   
                           -  H H_t \Psi 
                           + \I  H H_\rho [H,\rho] \Psi 
                            - 2  H_t H \Psi   \\&
                            + 2\I  H_\rho [H,\rho] H \Psi  
                           -\I H_{tt} \Psi
                           + 2 (H_{t})_\rho [H,\rho]  \Psi
                           +\I H_{\rho\rho} ([H,\rho])^2  \Psi \\&
                           - H_\rho [H_t,\rho] \Psi
                           +\I H_\rho [H_\rho [H,\rho],\rho]  \Psi
                           +\I H_\rho  [H,[H, \rho]] \Psi.
\end{align*}
Taking the norm and applying \eqref{eqn:lte_ineq}, we obtain the desired result.

\end{proof}

\cref{lem:der_P}, \cref{lem:der_rho} and \cref{lem:LTE_2nd_schd} give the local truncation error errors of both PT and Schr\"odinger dynamics. Following the standard stability analysis \cite{LeVeque2007}, we obtain the global error bounds.

\begin{thm}[Global error] \label{thm:global_err}
For the IM schemes of \eqref{eqn:tddft} and \eqref{eqn:pt_mix} up to the time $t_n = T$, there exists some constant $C$ depending on $T$ and $\norm{H}$ such that
\begin{enumerate}
        \item the errors for the PT dynamics \eqref{eqn:pt_mix} satisfy  
\begin{equation}
\norm{\Phi(t_n) - \Phi_n} + \norm{\sigma(t_n) - \sigma_n} \leq C f_1(H,\rho, P) h^2, \\
\end{equation}
where $f_1$ is a function of $H, \rho, P$ that is a linear combination of products of nested commutators up to three layers of the form
\begin{equation}
 \norm{[A_4,  A_3[ A_2, A_1 A_0]]}
\end{equation}
with $A_0$ being one of the following
\begin{equation}
[H, P], \quad [H_t, P], \quad [H_{tt}, P], \quad [H, \rho], \quad [H_t, \rho], \quad [H_{tt}, \rho]
\end{equation}
and $A_i$ ($i = 1, \cdots, 4$) being the identity matrix $I$, functions of $H, \rho, P$ or derivatives of $H$.
\item the error for the Schr\"odinger dynamics \eqref{eqn:tddft} satisfies
\begin{equation}
 \norm{\Psi(t_n) - \Psi_n}  \leq C \left( f_2(H,\rho, P) + \norm{H^3 \Psi} +  \norm{H H_t \Psi} + 2\norm{H_t H \Psi} + \norm{H_{tt} \Psi}\right) h^2,
\end{equation}
where $f_2$ has the same form as $f_1$. 
\end{enumerate}

\end{thm}

\cref{thm:global_err} shows that the error bound of PT dynamics exhibits commutator scaling while that of the Schr\"odinger equation does not. 
We remark that the worst-case dependence of the constant $C$ on the norm of $H$ can be very pessimistic, which is due to the standard stability analysis through the Gr\"onwall type estimates. However, the Schr\"odinger equation inherits a Hamiltonian structure and, together with the fact that IM is a symplectic scheme, this preconstant $C$ may be dramatically improved such that it depends linearly on $T$ and is even possibly independent of $\|H\|$ \cite{HairerLubichWanner2006}. 
In order to formally employ the symplectic properties, however, the PT-IM scheme needs to be slightly modified.
This has been demonstrated in \cite{AnLin2020} for pure states. 
Numerical results in \cite{AnLin2020} also demonstrate that the performance of the schemes with and without the modification are almost the same, so the modification may only be of theoretical interest.
For simplicity, we do not detail such modification here.

\subsection{Near adiabatic regime} \label{sec:near_adia}

In the near adiabatic regime, we can use commutator structure to demonstrate provable advantage of the PT dynamics over the Schr\"odinger dynamics. 
Consider the singular perturbed Schr\"odinger equation:
\begin{equation}\label{eqn:schd_singular}
i\eps \partial_t \Psi^\eps(t) = H(t, \rho^\eps(t)) \Psi^\eps(t), \quad \eps \ll 1.
\end{equation}
Here $\rho^\eps(0)$ is a pure state, and $\Psi^\eps(0)$ consists of the eigenfunctions of $H(0,\rho^\eps(0))$ corresponding to the algebraically lowest $N$ eigenvalues.

Let $\rho^\eps = P^\eps = \Psi^\eps \Psi^\eps\,^\dagger$. Then the PT dynamics become
\begin{equation}
  \I \eps \partial_t \Phi^\eps(t) = H(t,\rho^\eps(t))\Phi^\eps(t) -
  \Phi^\eps(t)(\Phi^\eps\,^{\dag}(t)H(t,\rho^\eps(t)) \Phi^\eps(t)), \quad \Phi^\eps(0)=\Psi^\eps(0).
\end{equation}


In the linear case ($H(t,\rho(t))=H(t)$ is independent of $\rho$), if the gap condition is satisfied, i.e. there exists a positive gap between the $N$-th and $(N+1)$-th smallest eigenvalues of $H(t)$ for all $t\in[0,T]$, 
The adiabatic theorem (see for example, \cite{Teufel2003,HagedornJoye2002,Joye1995,Joye2007}) for the Schr\"odinger dynamics \eqref{eqn:schd_singular} states that
\begin{equation} \label{eqn:adia_thm_schd}
\Psi^\eps(t) = \Psi_a(t) + \Or(\eps),
\end{equation}
where the columns of $\Psi_a(t)$ are the eigenvectors of the Hamiltonian, namely, there exists a time-dependent diagonal matrix $\Lambda(t)$ whose diagonal entries are eigenvalues of the Hamiltonian such that 
\[
H(t) \Psi_a(t) = \Psi_a(t) \Lambda(t). 
\]
The adiabatic theorem can also be generalized to certain linear systems without a gap condition \cite{AvronElgart1999,Teufel2001}, and for some weakly nonlinear systems~\cite{Fermanian-KammererJoye2020,GangGrech2017}. A detailed discussion of the technical conditions for the adiabatic approximation is beyond the scope of this paper. However, when such \textit{a priori} estimate is available, we can evaluate the commutator as
\begin{equation} \label{eqn:comm_eps}
[H, \rho^\eps] = H \Psi_a \Psi_a^\dagger - \Psi_a \Psi_a^\dagger H + \Or(\eps) = \Psi_a \Lambda \Psi_a^\dagger - \Psi_a \Lambda \Psi_a^\dagger  + \Or(\eps) =\Or(\eps).
\end{equation}
We now examine the commutator terms in \cref{lem:der_rho}. Note that in the singular perturbed regime, one should replace the $H$ in \cref{lem:der_rho} by $H/\eps$, and hence the leading order terms in $\eps$ are
\begin{align*}
 & \eps^{-3}\norm{[H_{\rho^\eps\rho^\eps} ([H,\rho^\eps])^2, \rho^\eps] } 
 + \eps^{-3} \norm{ [H_{\rho^\eps} [H_{\rho^\eps} [H,\rho^\eps],\rho^\eps], \rho^\eps]}  
 + \eps^{-3} \norm{ [H_{\rho^\eps}  [H,[H, \rho^\eps]], \rho^\eps]} 
 \\
 & + \eps^{-3} \norm{ [H_{\rho^\eps} [H,\rho^\eps], [H,\rho^\eps]] } 
 + \eps^{-3} \norm{ [H, [H_{\rho^\eps}[H,\rho^\eps] ,\rho^\eps]]} 
 + \eps^{-3} \norm{ [H, [H,[H,\rho^\eps]]]}
 = \Or(\eps^{-2}),
\end{align*}
thanks to \eqref{eqn:comm_eps}. 
However, by replacing $H$ in \cref{lem:LTE_2nd_schd} by $H/\eps$, we obtain $\norm{\partial_t^3 \Psi}  = \Or(\eps^{-3})$ for the Schr\"odinger dynamics. Finally applying \cref{thm:global_err}, we find that the numerical schemes for the PT dynamics can gain an order of magnitude in terms of the accuracy in $\eps$, which recovers the result in \cite{AnLin2020} for the linear case, and generalizes the result to the nonlinear case (provided that adiabatic theorems can be established).

\section{Numerical Results} \label{sec:num_results}

In this section, we provide the numerical results of the parallel transport dynamics. Extensive numerical results and implementations on real chemical systems have been demonstrated for the PT dynamics with pure initial states~\cite{AnLin2020,JiaAnWangEtAl2018,JiaLin2019,JiaWangLin2019}. Hence, we focus on the case of a mixed initial state in this section.
In numerical examples, the relative numerical errors are computed by 
\begin{equation*}
\sup_{0 \leq k \leq n} 
\frac{\|X_k - X(t_k) \|}{\|  X(t_k) \|} ,
\end{equation*}
where 
$k = 0, 1, \cdots, n$ and $t_n$ is the final time, and 
$X$ is the reference values of $\Phi$, $\sigma$ or $\Psi$ with $X_n$ represents the numerical results of the quantity $X$ at the time $t_n$.

Our model system is defined by a periodic potential field given by a one-dimensional lattice structure with Hamiltonian 
\begin{equation} \label{eq:W}
H(t) = -\frac{1}{2}\Delta + V(x) + W(x,t).
\end{equation}
Here $V(x) = \cos(x)$ is a static potential. The external time-dependent potential with frequency $\omega$ is
\begin{equation} \label{eqn:Wt}
W(x,t)  = 10 \sin\left(\frac{x}{L}\right) \sin(\omega t), 
\end{equation}
and $L$ denotes the number of unit cells. The length of the lattice (the computational domain) is $2\pi L$. Fig. \ref{fig:v_w} shows a typical plot for the two potentials over the lattice cells.
The parameters in the system are chosen as $L = 4$, $\beta = 1.453$, $\omega = 16\pi$ and the chemical potential $\mu = 3.299$. The initial occupation number according to the Fermi-Dirac distribution is in \cref{fig:initial_cond}. Each unit cell is discretized via $64$ equidistant grid points, and hence the total number of grid points is $N_g=64L$.  
%
%
%

We first verify the \cref{prop:ortho} numerically by simulating the PT dynamics to $T_\text{final} = 4$ using the PT-IM scheme with a step size $h = 0.01$. We set $N_e=20$, and $N=64$. 
The norm of $\Phi_{n+1}^\dagger \Phi_{n+1}$ and values of $\Tr{\sigma_{n}}$ and $\Tr{\sigma^2_{n}}$ are plotted for the simulation time in \cref{fig:verf_prop}. It can be seen that the values of all three quantities are constant throughout the simulation, which agrees with \cref{prop:ortho}. In comparison, we also plot the higher order trace $\Tr{\sigma_n^3}$, which is not a conserved quantity. Nonetheless, the fluctuation of $\Tr{\sigma_n^3}$ is still very small and on the order of $10^{-6}$.

\begin{figure}[htbp]
\centering
\subfloat[the potentials]{
\includegraphics[scale=0.4]{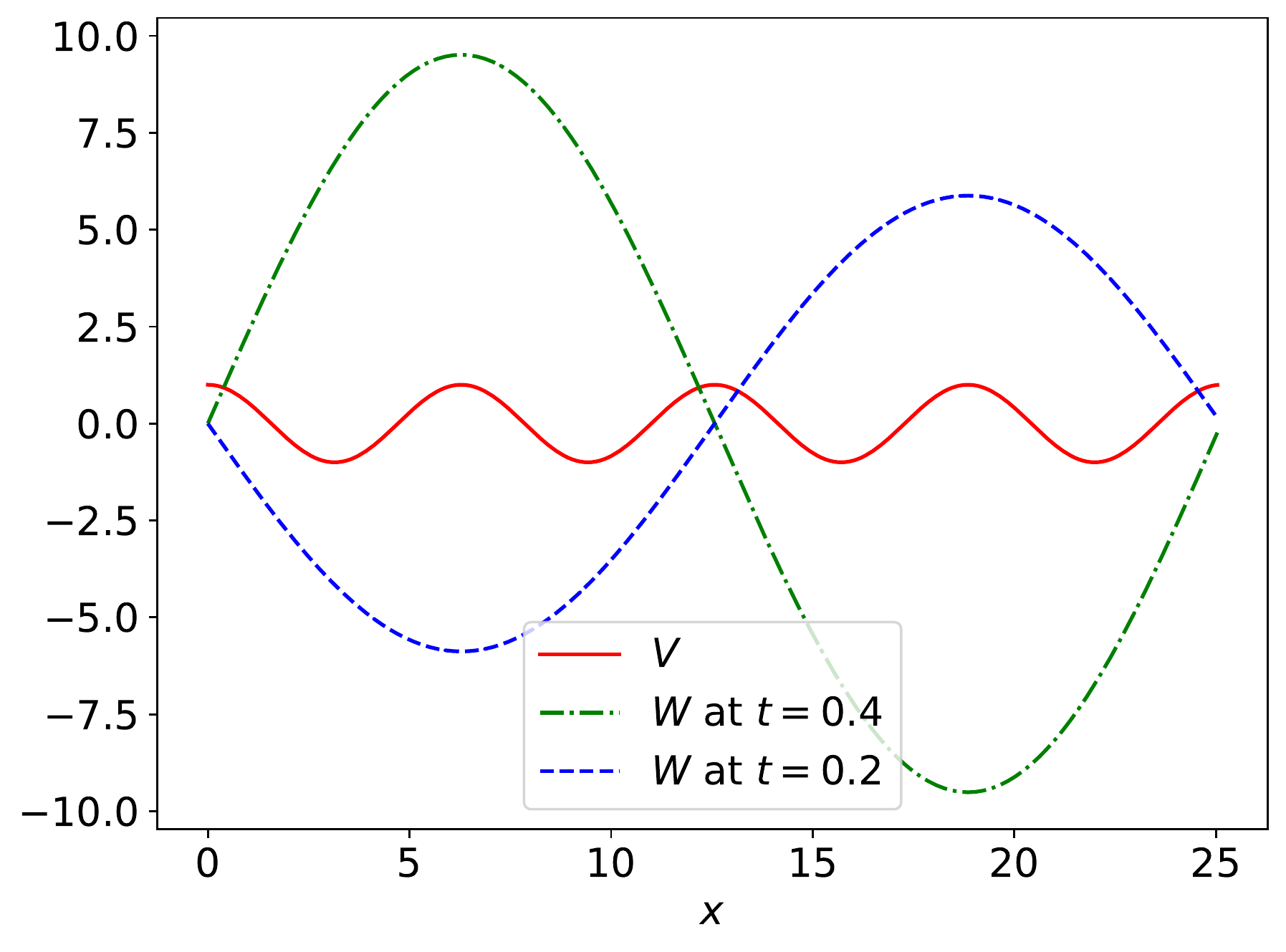}  \label{fig:v_w}}
\subfloat[the initial occupation]{
\includegraphics[scale=0.4]{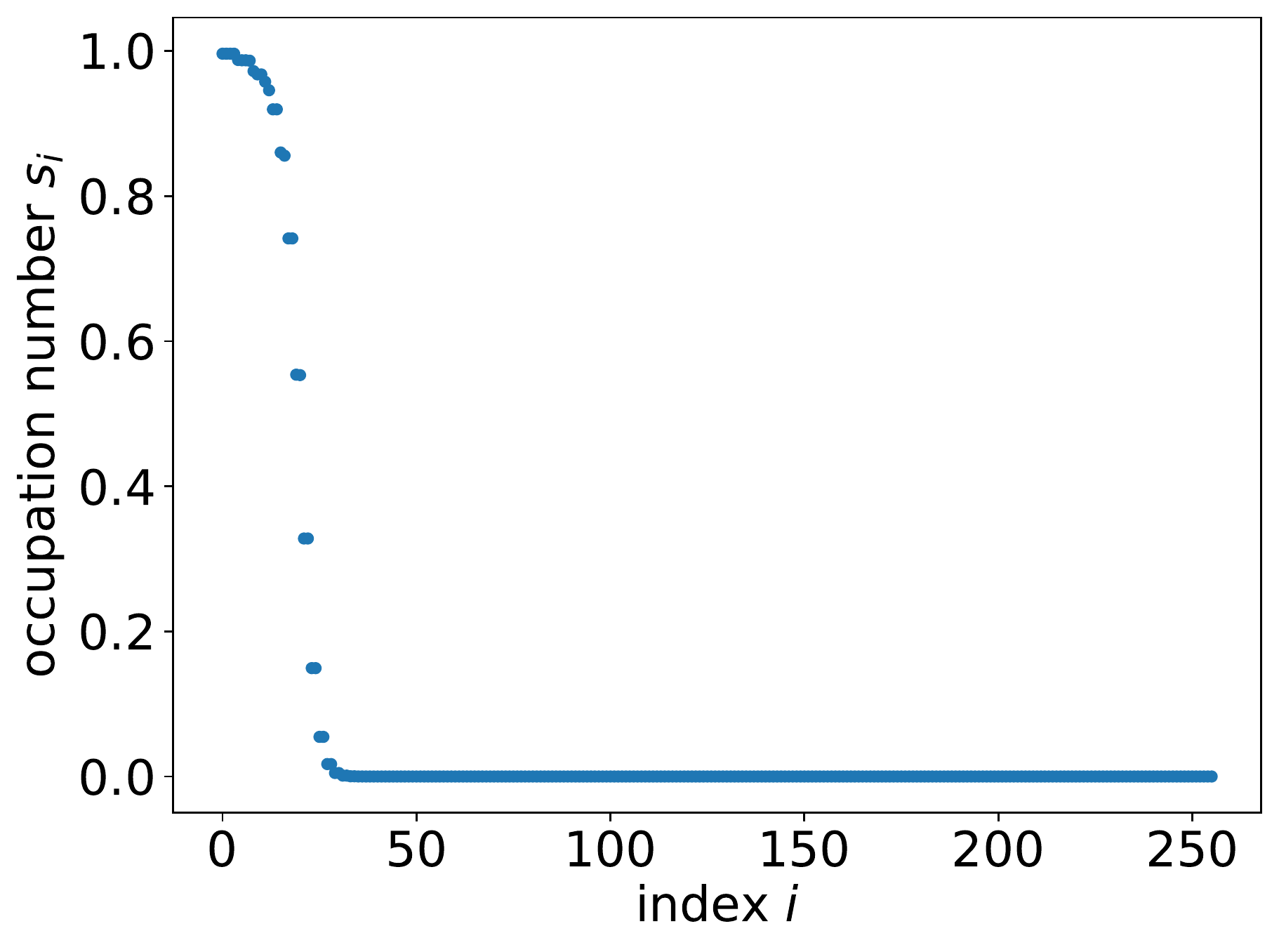}
\label{fig:initial_cond}
}

\caption{Left panel: The potentials $V(x)$ (red solid) and $W(t,x)$ at time $t = 0.2$ (blue dashed) and $t = 0.4$ (green dotted), respectively, where $W$ is of time period $1/8$. Right panel: The initial occupation of the Fermi-Dirac statistics. $L = 4$, $\beta = 1.453$, and the chemical potential is chosen such that the initial number of occupation $N_e = \Tr(\rho(0)) = 20$.}
\end{figure}

\begin{figure}[H]
\centering
\subfloat[$\norm{\Phi^\dagger \Phi - I}$]{
\includegraphics[scale=0.33]{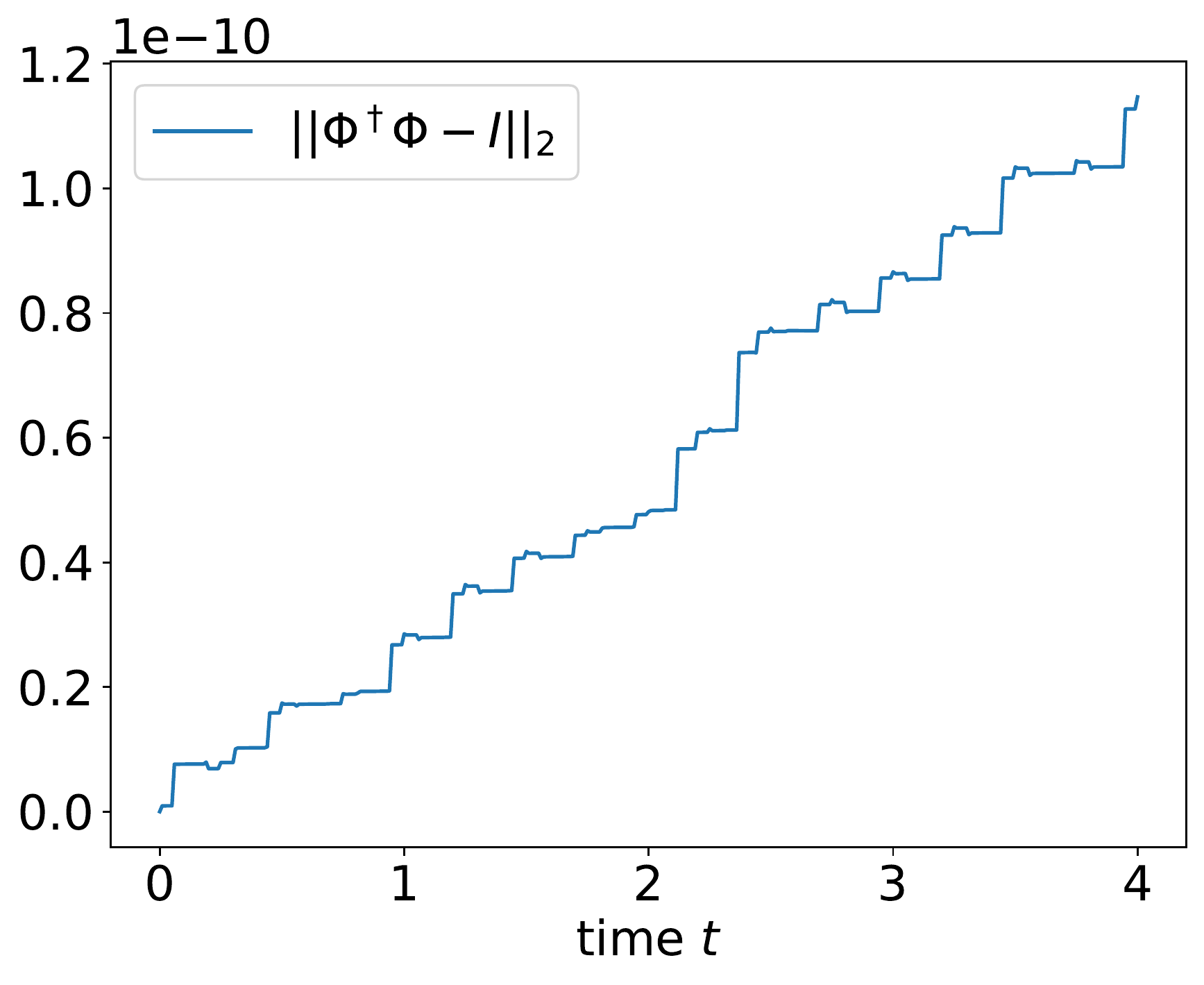}}
\subfloat[$\Tr{\sigma}$]{
\includegraphics[scale=0.33]{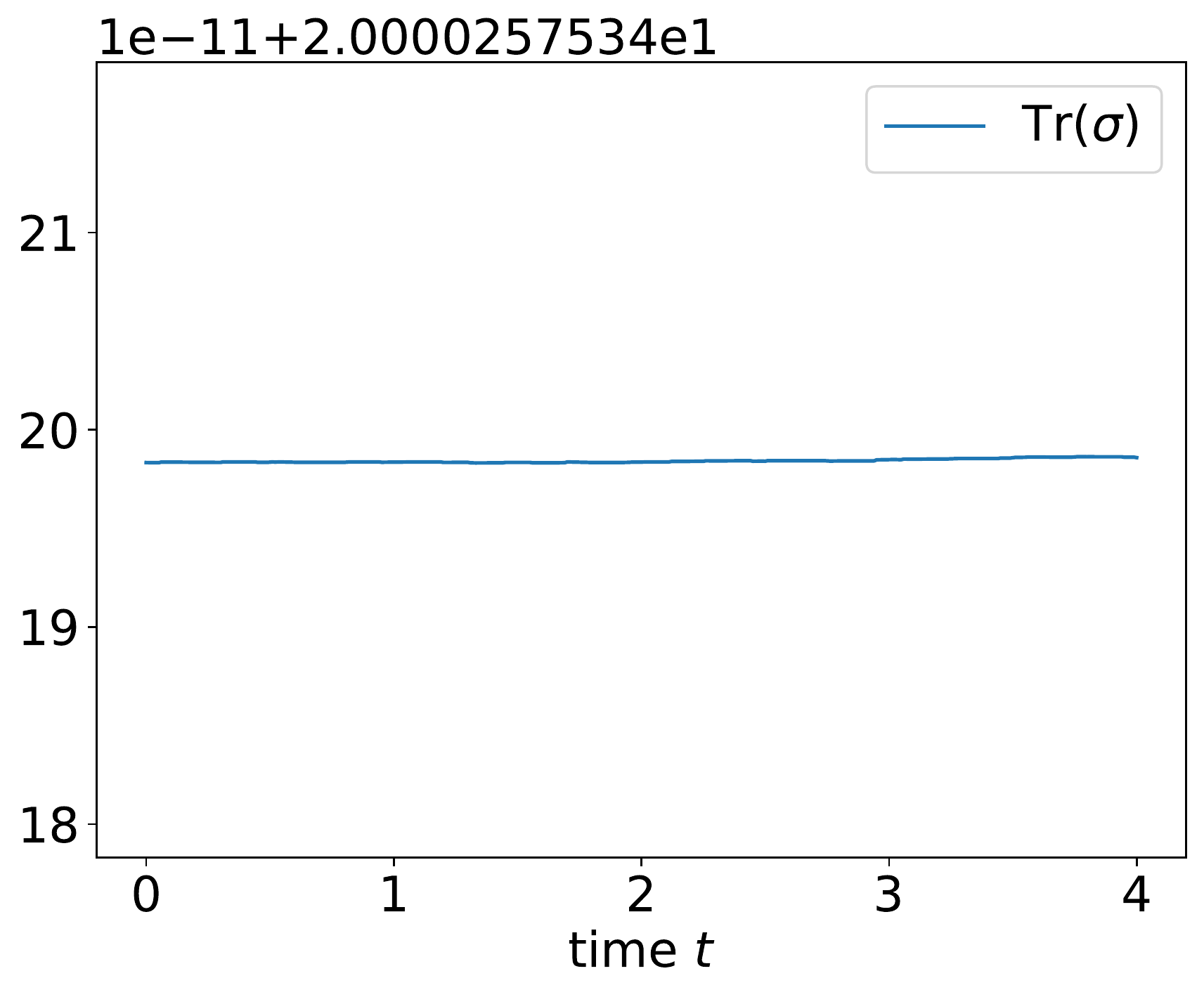}}
\\
\subfloat[$\Tr{\sigma^2}$]{
\includegraphics[scale=0.33]{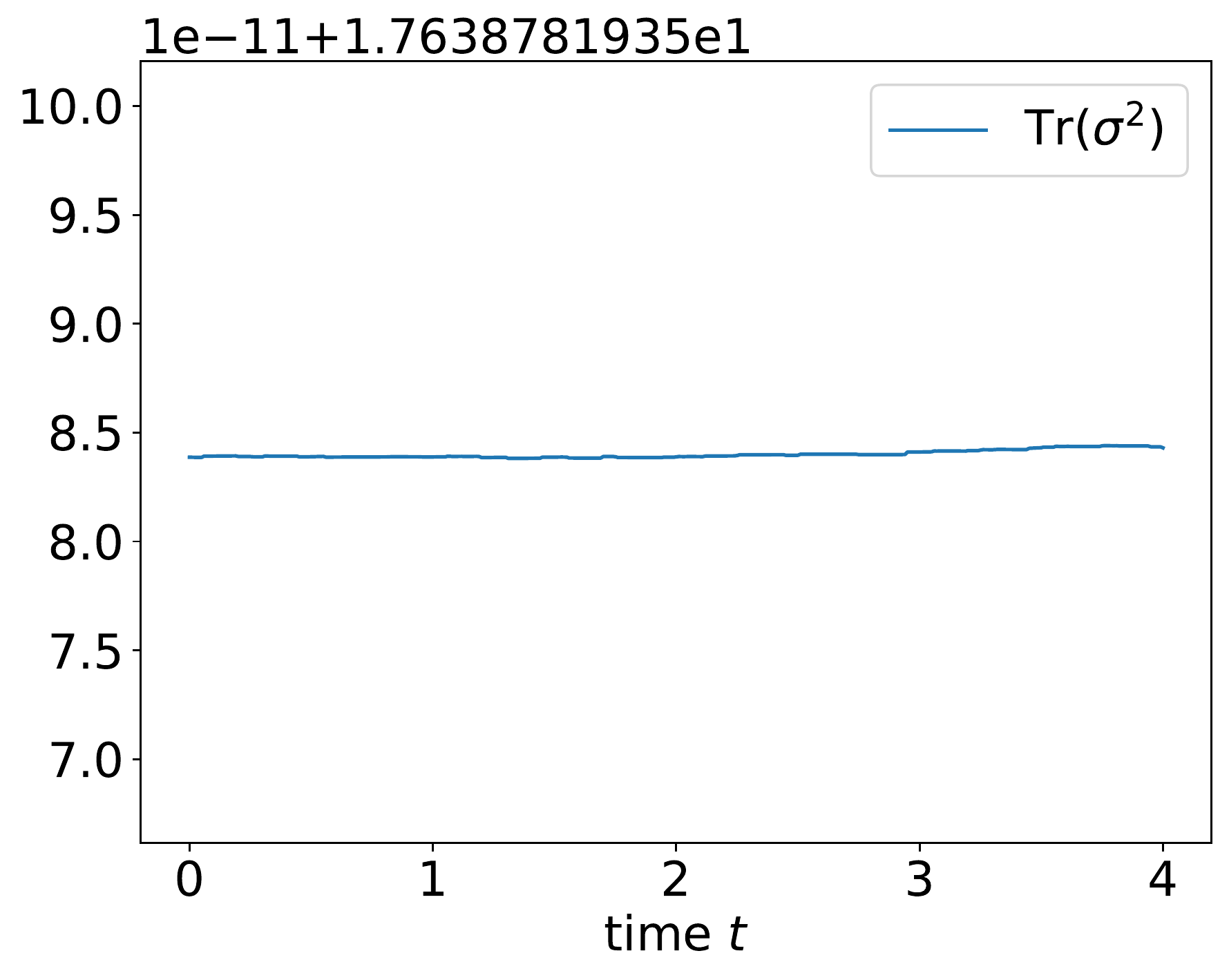}}
\subfloat[$\Tr{\sigma^3}$]{
\includegraphics[scale=0.33]{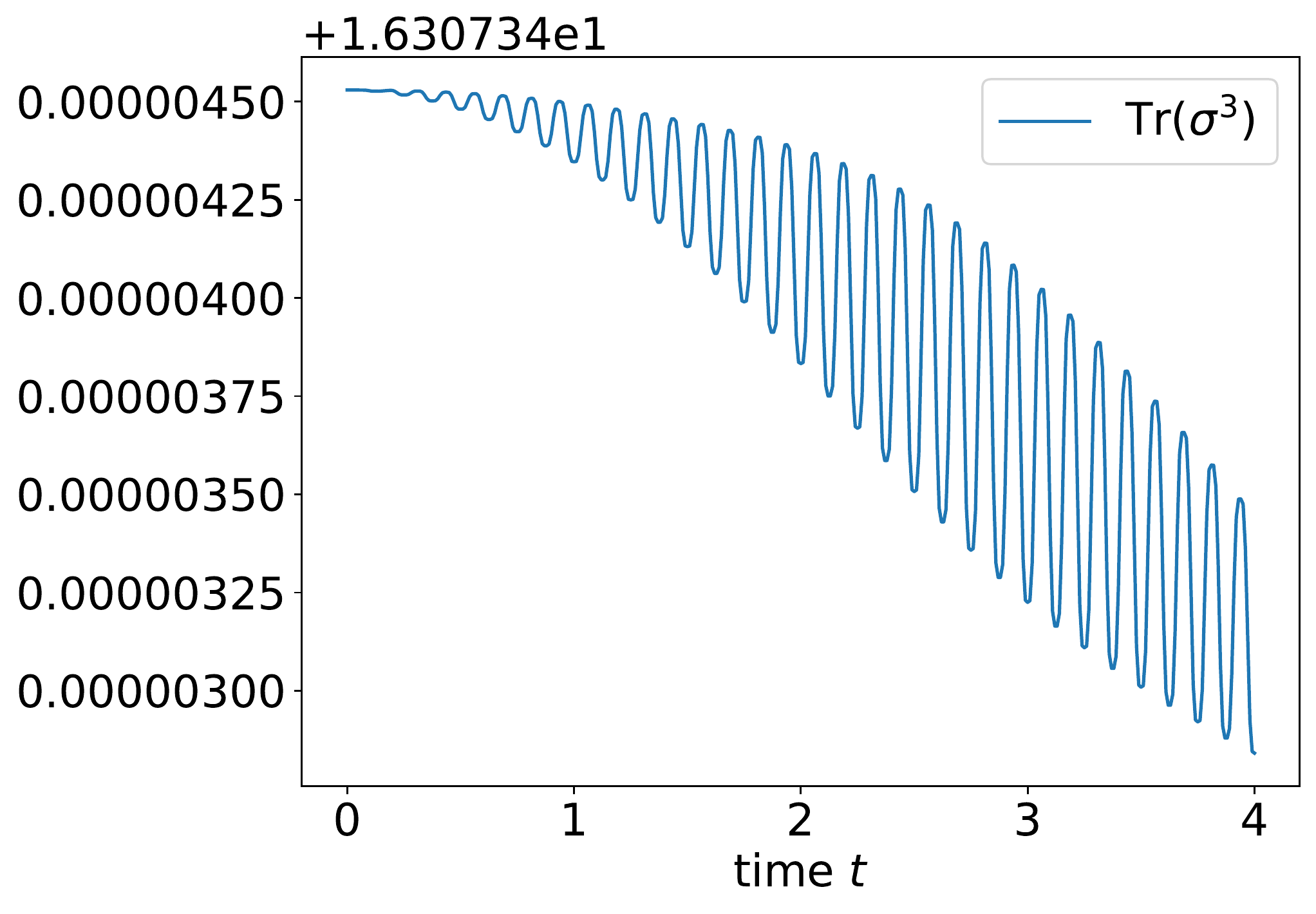}}
\caption{
Numerical verification (time step $h = 0.01$) of the orthogonality of $\Phi$ and the trace preservation of $\sigma$ and $\sigma^2$, as shown in \cref{prop:ortho}. On the other hand, the trace of higher powers of $\sigma$ (e.g. $\sigma^3$) may not be preserved in the PT-IM scheme.}
\label{fig:verf_prop}
\end{figure}

Next, we compare the numerical errors in simulating the Schr\"odinger dynamics (SD) and PT dynamics. Both dynamics are simulated using IM schemes to $T_\text{final} = 1$. We set $\mu = 26.893$ (corresponding to $N_e = 60$) and $N = 80$.
In order to verify the convergence rate numerically, we set the time steps to be $0.05$, $0.02$, $0.01$, $0.005$, $0.002$, $0.001$. The reference solution is computed using a fine time step of $2\times 10^{-5}$. \cref{fig:error_dt} shows that both SD-IM and PT-IM are second order methods, but the preconstant of PT-IM is much smaller. 
The accuracy of the PT dynamics can also be shown in terms of physical observables, e.g. the dipole moment:
\[
\langle x(t) \rangle := \Tr{(x \rho(t))}.
\]
\cref{fig:dip} compares the dipole moment computed in three different scenarios: PT-IM with $h = 0.02$, SD-IM with $h = 0.02$, and SD-IM using a very small time step $h = 0.0001$. We find that the difference between the time-dependent dipole moment obtained from PT-IM with a large time step $h = 0.02$ is almost the same as that from the reference solution. However, SD-IM with the same time step size is only accurate for a short periodic of time, and its accuracy significantly deteriorates as $t$ increases.

\begin{figure}[H]
\centering
\subfloat[Relative Errors]{
\includegraphics[scale=0.35]{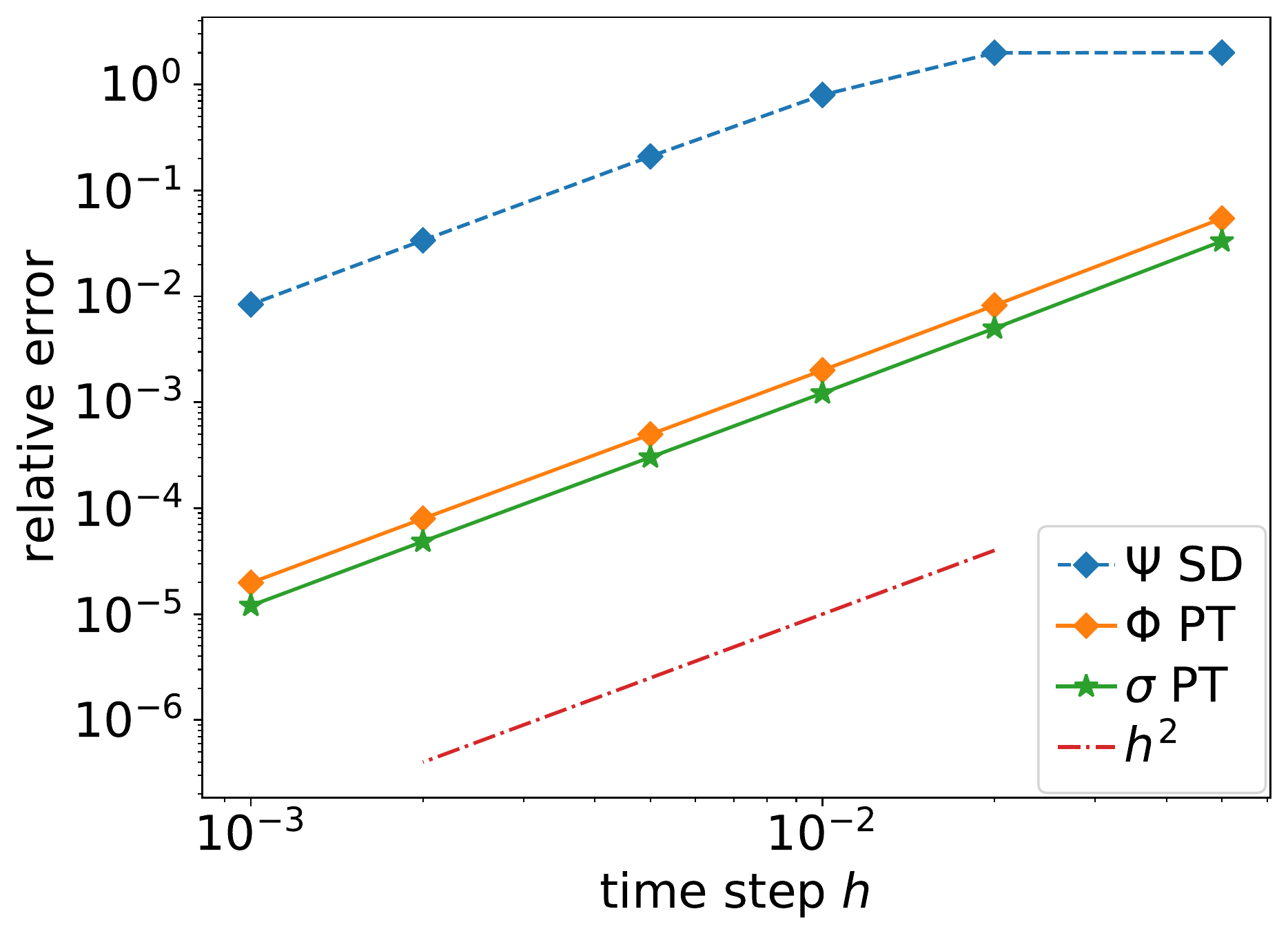} \label{fig:error_dt}}
\subfloat[Dipole moment]{
\includegraphics[scale=0.35]{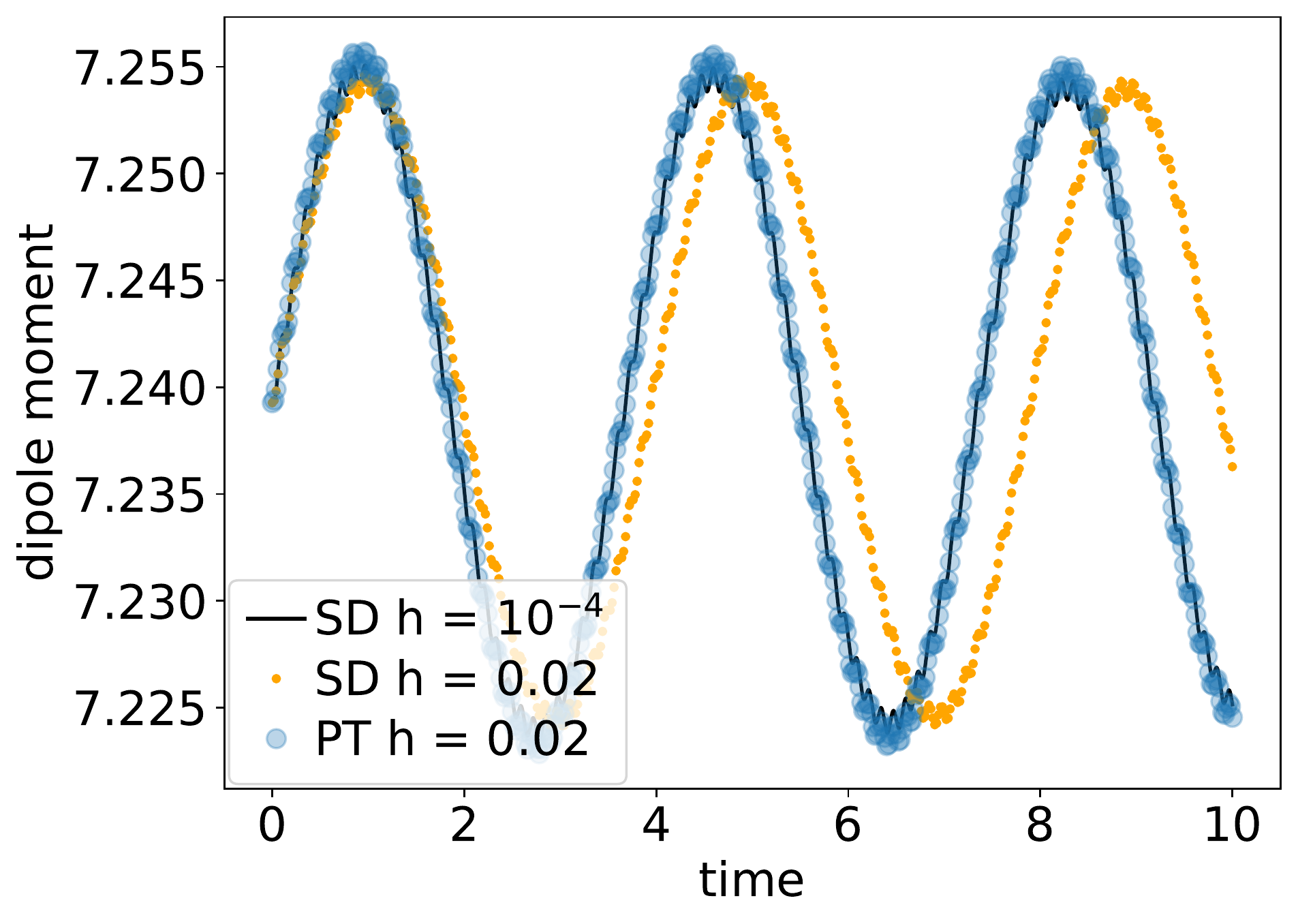} \label{fig:dip}}
\caption{Left Panel: Log-log plot of the relative errors of $\Phi$, $\sigma$, $\Psi$ and the density matrix $\rho$ computed via both PT and Schr\"odinger dynamics (SD). Right Panel: Evolution of the dipole moment for PT-IM with $h = 0.02$, SD-IM with $h=0.02$, and SD-IM with $h = 0.0001$ (reference solution).}
\end{figure}

In order to demonstrate that the commutator scaling in \cref{thm:global_err}, we now vary the number of electrons $N_e$, and compare the results of PT-IM and SD-IM. 
The chemical potential $\mu$ is set to $3.299$, $7.028$, $12.291$, $18.951$, $26.893$, and the corresponding $N_e$ are $10$, $20$, $30$, $40$, $50$, $60$, respectively. 
We also set $N = N_e + 20, h = 0.01$, and $T_\text{final} = 1$. The reference solution is computed using a very small step size $h = 2\times 10^{-5}$.
 
We plot the relative errors of both the PT and the Schr\"odinger dynamics in comparison with our theoretical bounds. 
It can be seen in \cref{fig:ref_err_ne} that as $N_e$ increases, the relative error of the wavefunction in the Schr\"odinger dynamics grows much faster than that in the PT dynamics.
\cref{fig:ref_err_ne} also plots the terms in the error bounds with or without the commutator structures, respectively.
We find that the term without commutator structures can be much larger in magnitude, and the qualitative trend of the growth of the error bound with respect to $N_e$ matches that of the error from the numerical simulation.

\begin{figure}[H] 
\centering
\subfloat[errors in wavefunctions]{
\includegraphics[scale=0.35]{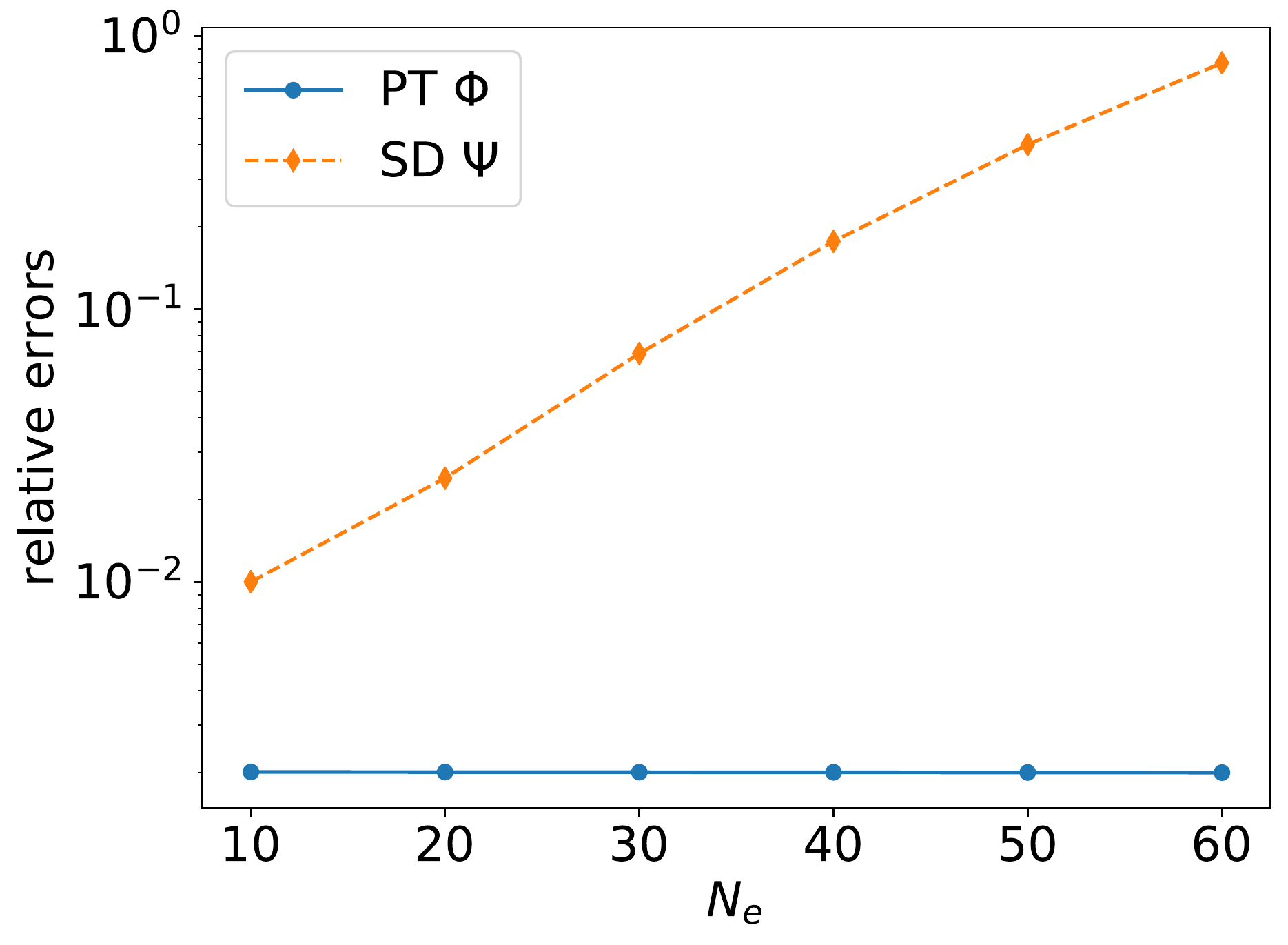} \label{fig:rel_err_u_ne}}
\subfloat[Commutator bounds and $\norm{H^3\Psi}h^2$]{
\includegraphics[scale=0.35]{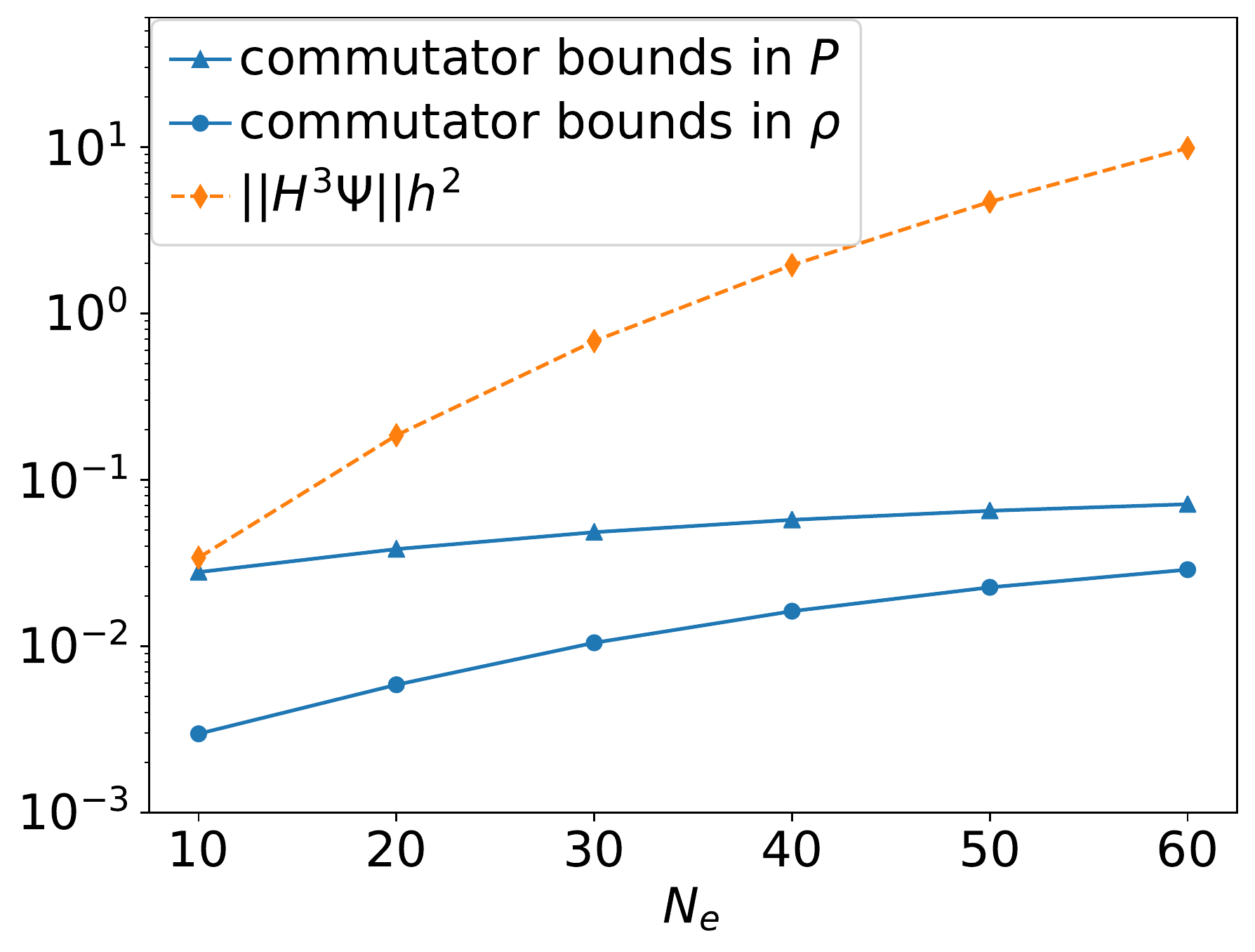} \label{fig:comm_vs_hpsi}}
\caption{Relative errors versus the number of occupation $N_e$ in semi-log scale. Left panel: relative errors for the wavefunctions $\Phi$ in PT gauge and $\Psi$ in the Schr\"odinger gauge. Right panel: the commutator bounds on the right-hand-sides of $\partial_t^3 P$ in \cref{lem:der_P} and $\partial_t^3 \rho$ in \cref{lem:der_rho} versus $\norm{H^3\Psi} h^2$ that appears in \cref{lem:LTE_2nd_schd}. 
The commutator bounds (as in PT) are significant smaller than $\norm{H^3\Psi}h^2$ term (as in the Schr\"odinger gauge). 
}
\label{fig:ref_err_ne}
\end{figure}

We also plot the relative errors in 2-norm of the density matrix $\rho$ in \cref{fig:rel_err_rho_ne_2norm} and \cref{fig:rel_err_rho_ne_fnorm}. The errors (measured in both the operator norm and the Frobenius norm) from the PT dynamics is smaller than that from the Schr\"odinger dynamics. 
Furthermore, as $N_e$ increases, the relative error in the Frobenius norm from the PT dynamics in fact decreases. This phenomenon can be intuitively explained as follows. Note that the initial $\sigma_0$ is a diagonal matrix of the following form
$$\begin{pmatrix}
I_{m_0} & 0 & 0 \\
0 & \sigma_* & 0 \\
0 & 0 & 0
\end{pmatrix},
$$
where $m_0$ is the number of fully occupied states and $\sigma_*$ is a diagonal matrix representing the fractional states whose diagonal elements has values in $(0,1)$. 
Then we expect that the fully occupied states are approximately in the near adiabatic regime, and their contribution to the error is much smaller than those from the fractionally occupied ones according to the commutator bound.
In this example, $m_0$ increases with $N_e$, but the size of $\sigma_*$ does not change much with respect to $N_e$. Therefore, the error of the density matrix should be dominated by a small number of orbitals near the Fermi surface.
To verify this statement, we plot in \cref{fig:his} the histogram of the errors in the vector 2-norm for all orbitals. Indeed, as $N_e$ increases, the errors are dominant by only a few orbitals near the corresponding chemical potential $\mu$, and the number of the orbitals with significant errors does not increase with $N_e$. 
On the other hand, the Frobenius norm of the density matrix $\norm{\rho}_F=\Or(\sqrt{N_e})$.
This explains the decay of the relative error of $\rho$ in the PT-dynamics in \cref{fig:rel_err_rho_ne_fnorm}.
By comparison, the histogram of the errors in the vector 2-norm for all orbitals in the Schr\"odinger gauge is provided in \cref{fig:his_schd}. We find that in the Schr\"odinger dynamics, the errors are propagated much more widely along the energy spectrum among a larger number of orbitals.
It is also interesting to note that the maximal magnitude of the error increases significantly with respect to $N_e$ in the Schr\"odinger dynamics, but the maximal error is nearly a constant and is much smaller in the PT dynamics.

\begin{figure}[htbp] \label{fig:ref_err_ne2}
\centering
\subfloat[errors of $\rho$ in 2-norm]{
\includegraphics[scale=0.35]{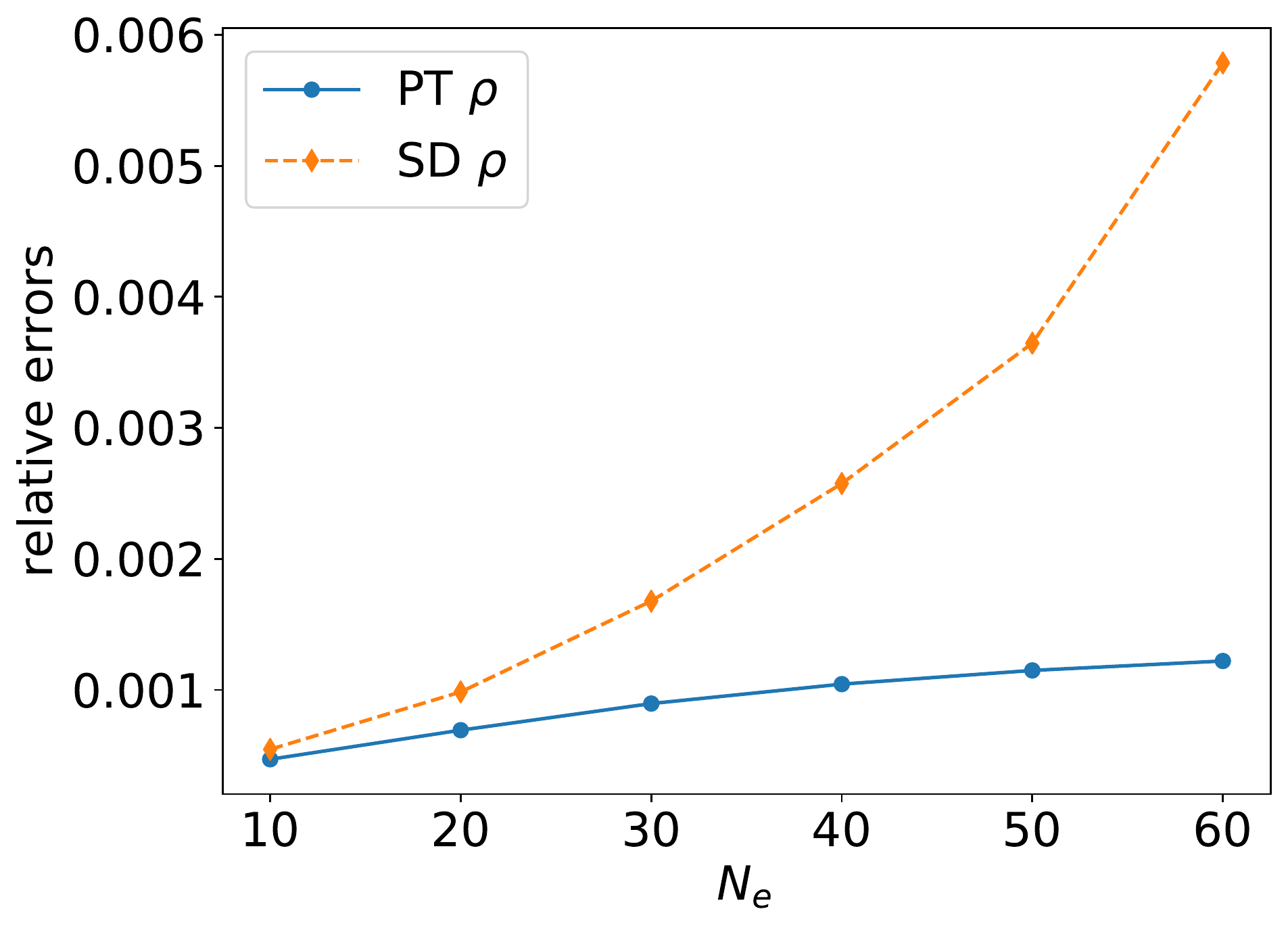} \label{fig:rel_err_rho_ne_2norm}}
\subfloat[errors of $\rho$ in F-norm]{
\includegraphics[scale=0.35]{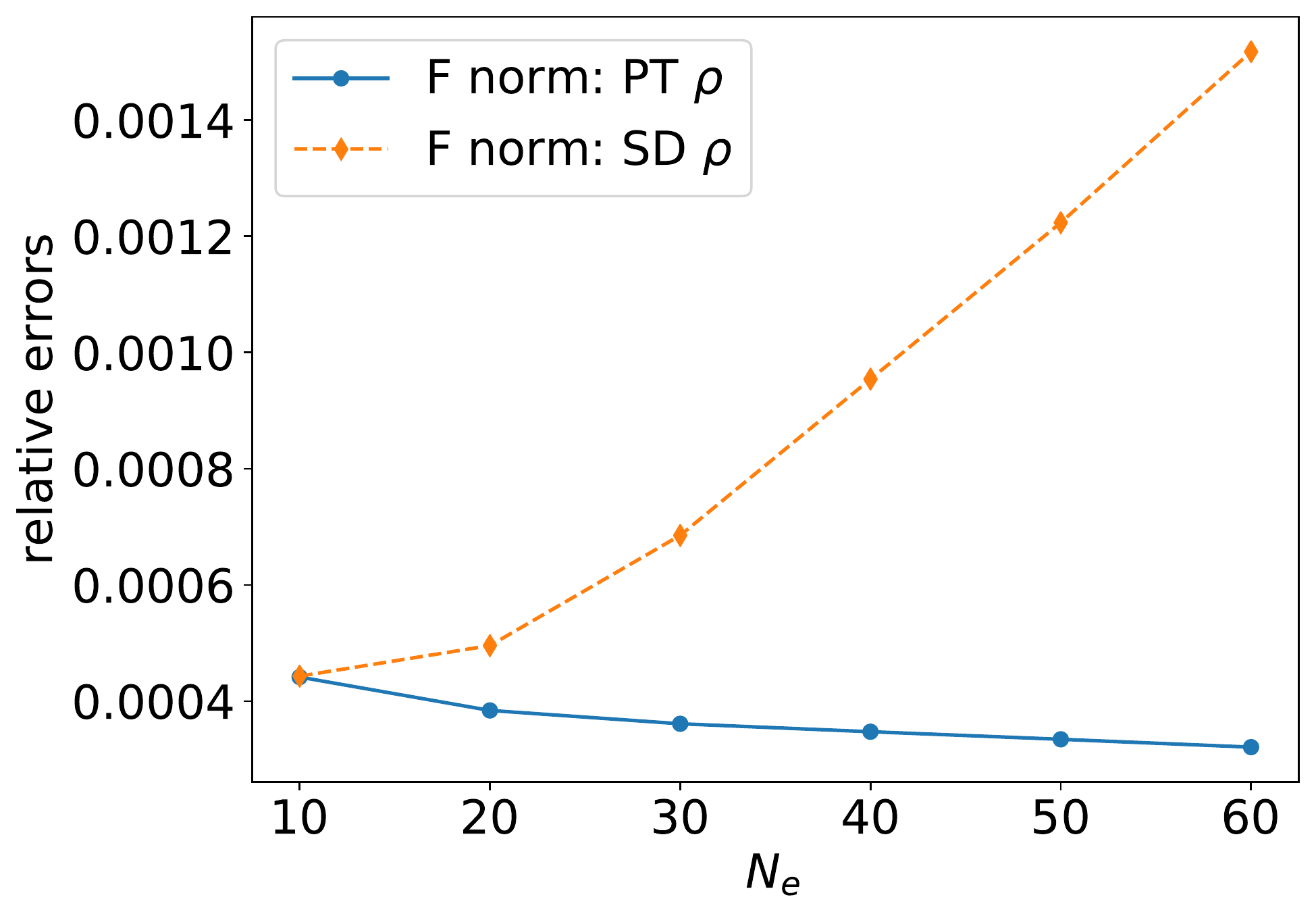} \label{fig:rel_err_rho_ne_fnorm}}
\caption{Plots of the relative errors of $\rho$ versus the number of occupation $N_e$. Left panel: relative errors in 2-norm. Right panel: the relative errors in the Frobenius norm (F-norm). 
}
\end{figure}

\begin{figure}[htbp] 
\centering
\subfloat[$N_e = 10$]{
\includegraphics[scale=0.25]{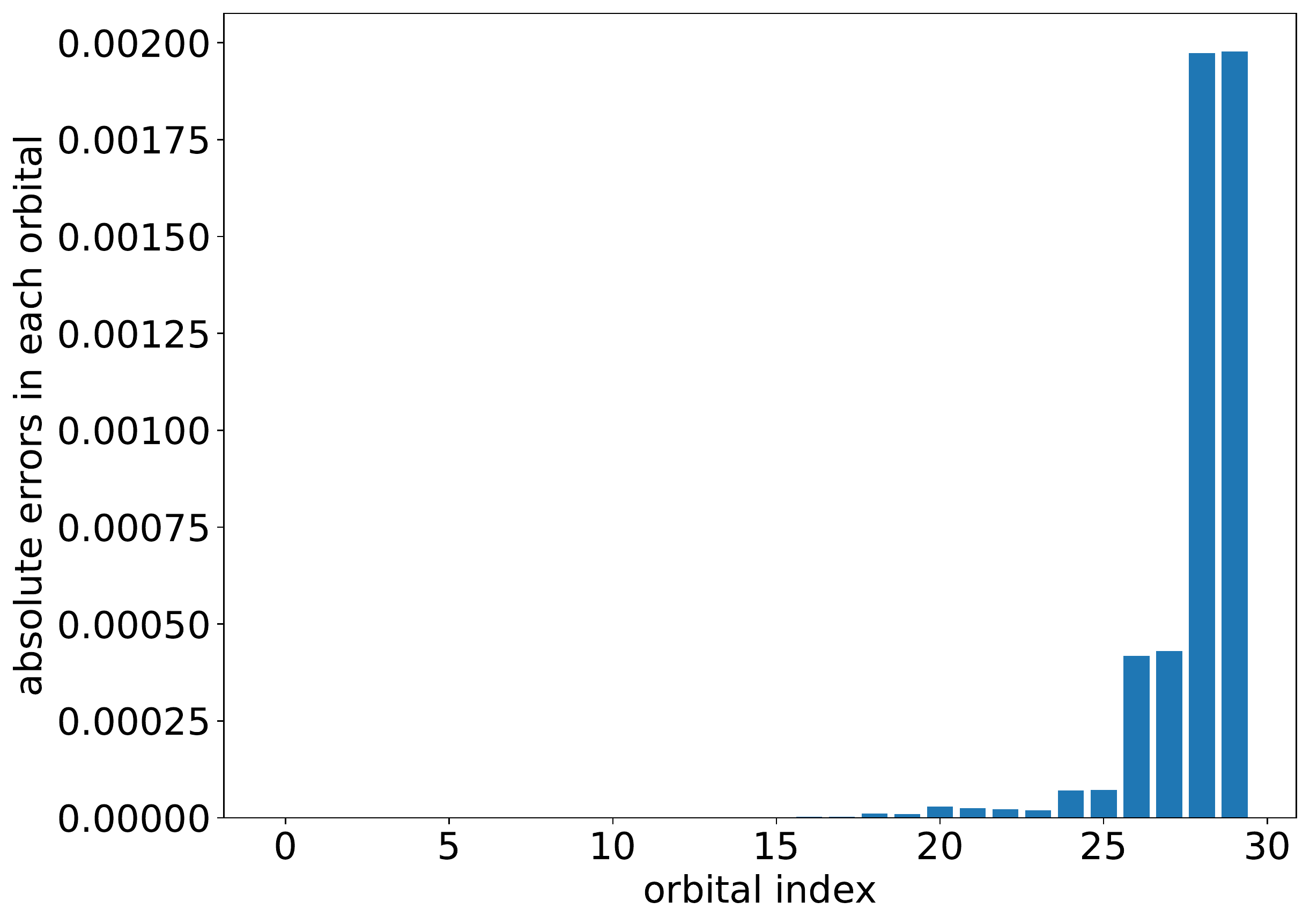} }
\subfloat[$N_e = 20$]{
\includegraphics[scale=0.25]{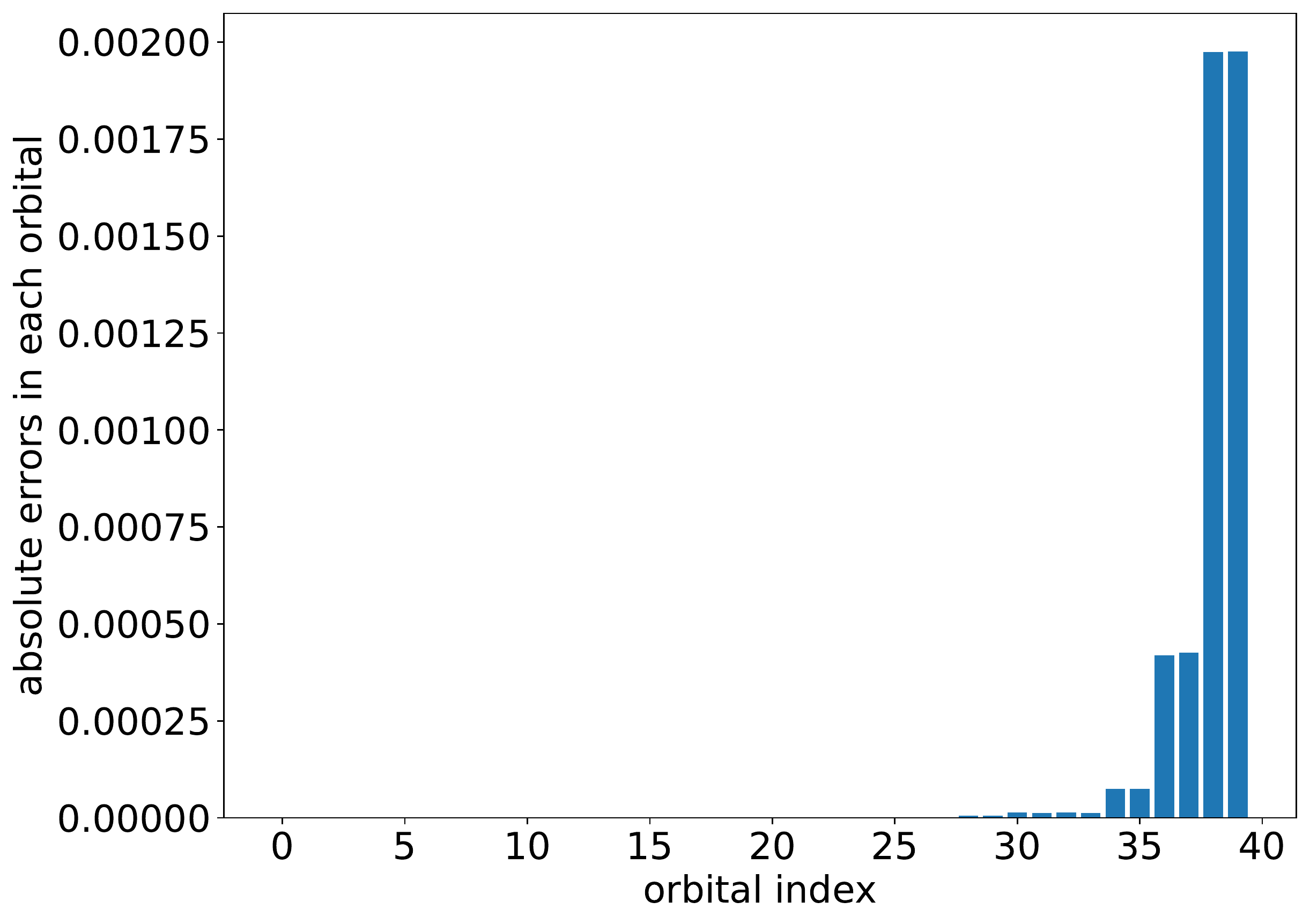} }
\\
\subfloat[$N_e = 30$]{
\includegraphics[scale=0.25]{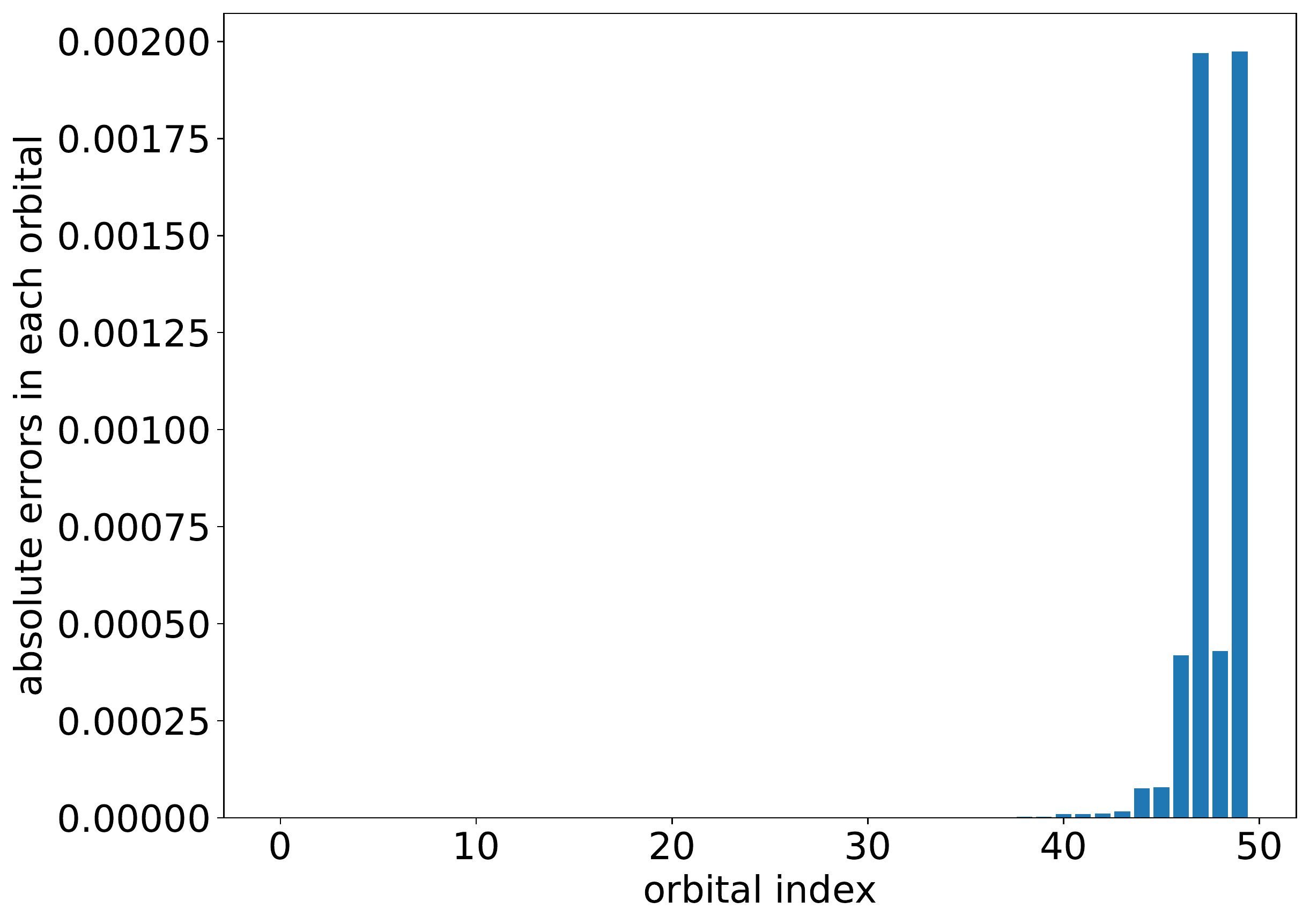} }
\subfloat[$N_e = 40$]{
\includegraphics[scale=0.25]{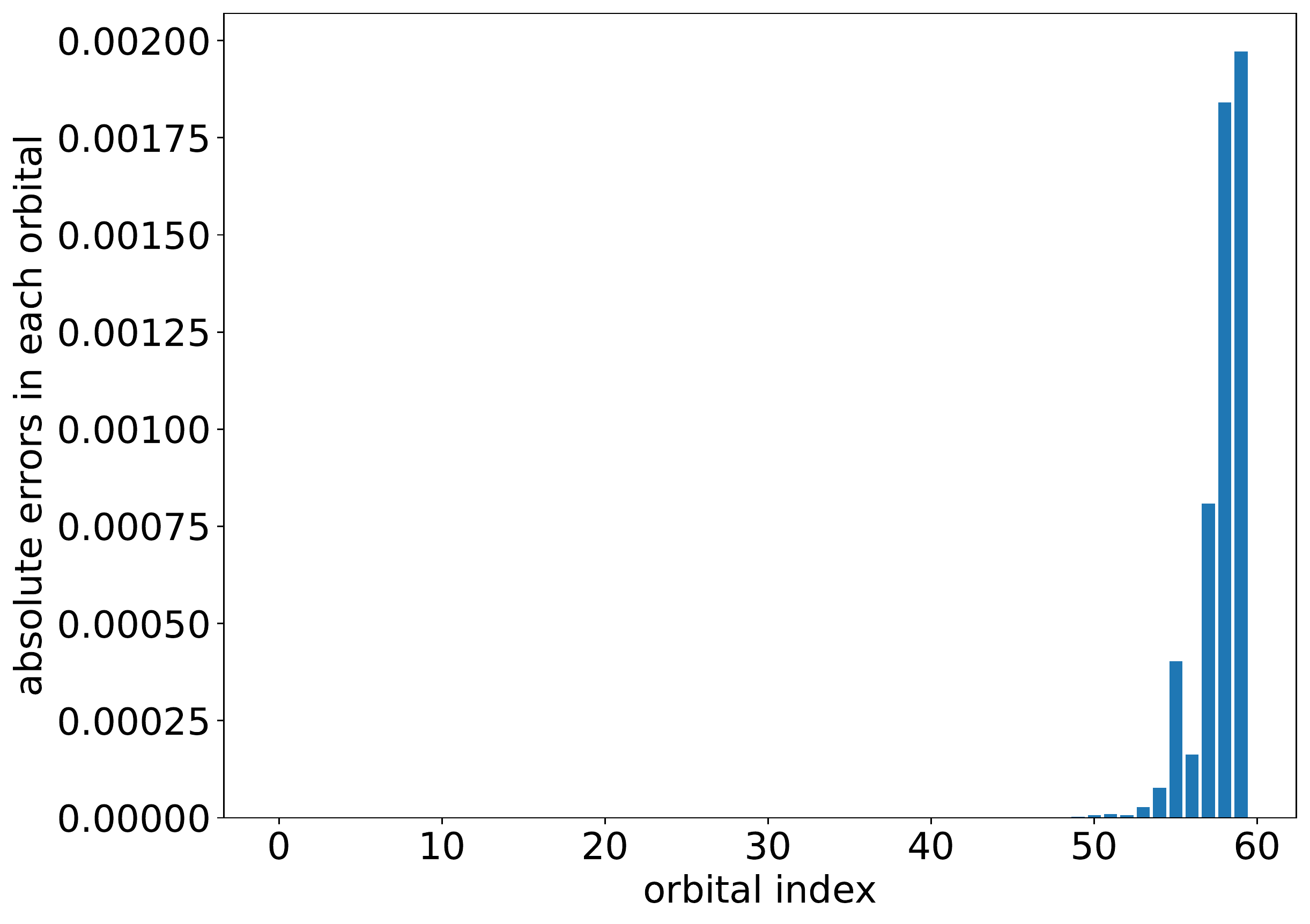} }
\\
\subfloat[$N_e = 50$]{
\includegraphics[scale=0.25]{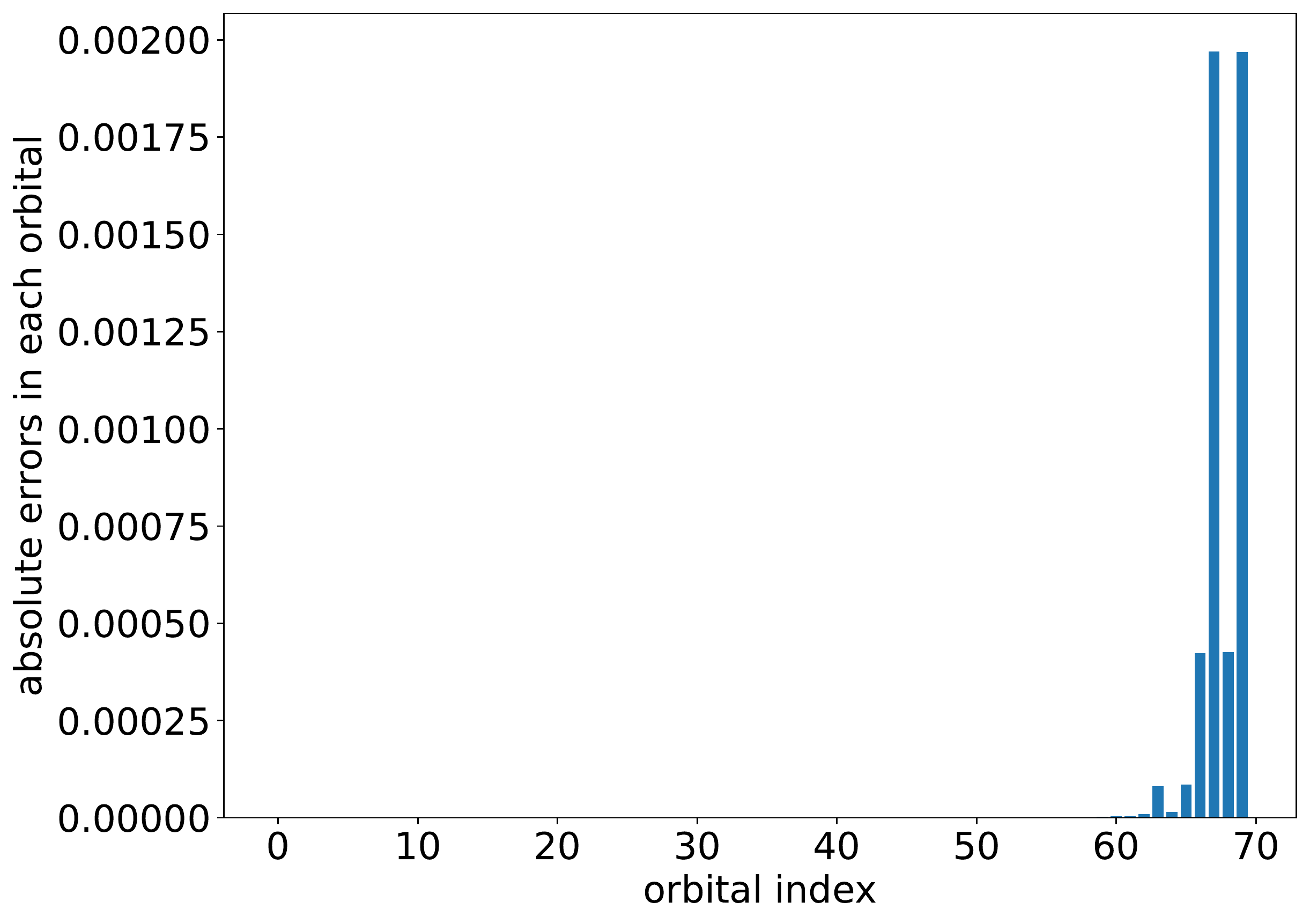} }
\subfloat[$N_e = 60$]{
\includegraphics[scale=0.25]{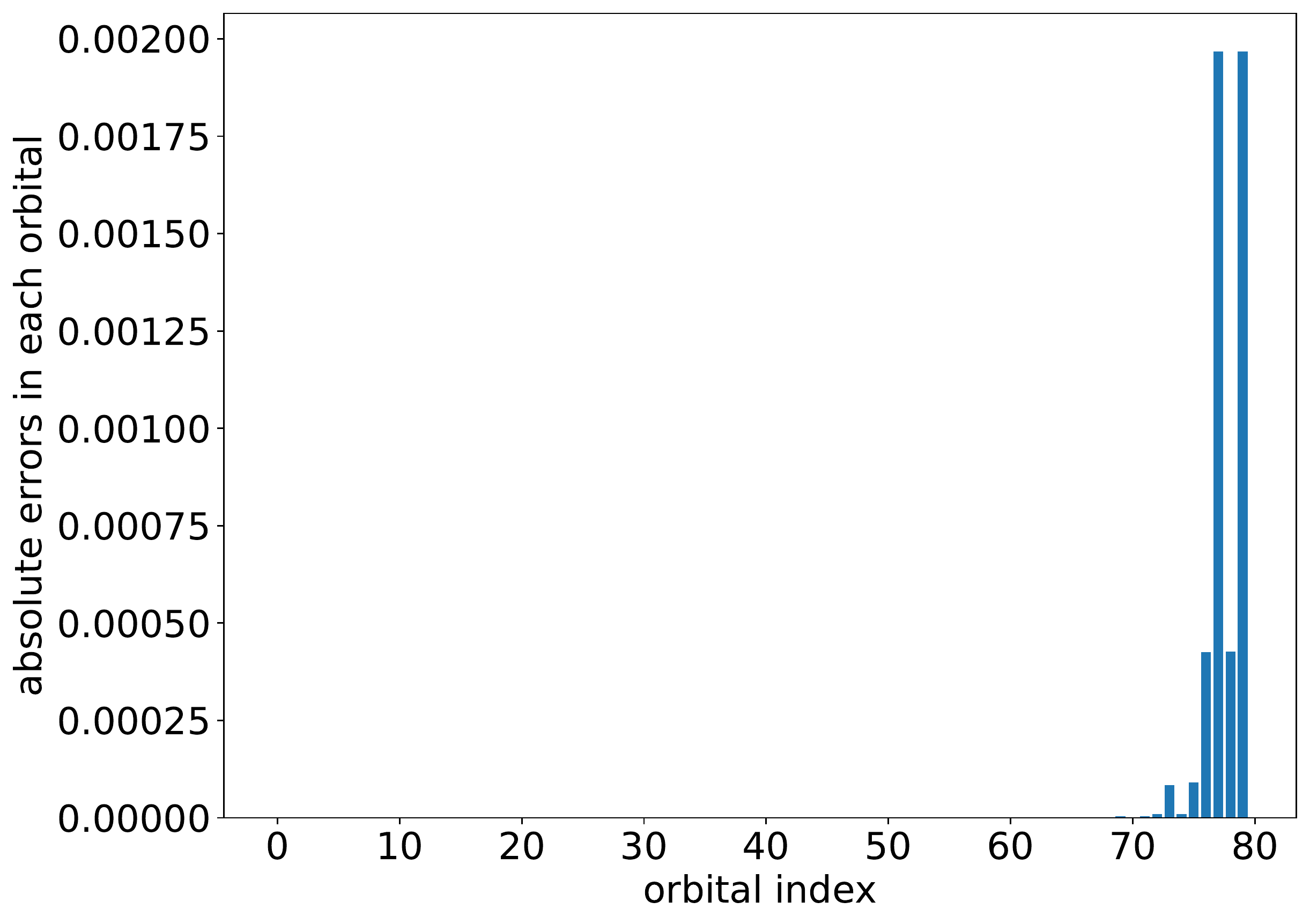} }
\caption{Plot of the error histogram of all orbitals for various $N_e$ in the PT dynamics.
}
\label{fig:his}
\end{figure}

\begin{figure}[htbp] 
\centering
\subfloat[$N_e = 10$]{
\includegraphics[scale=0.25]{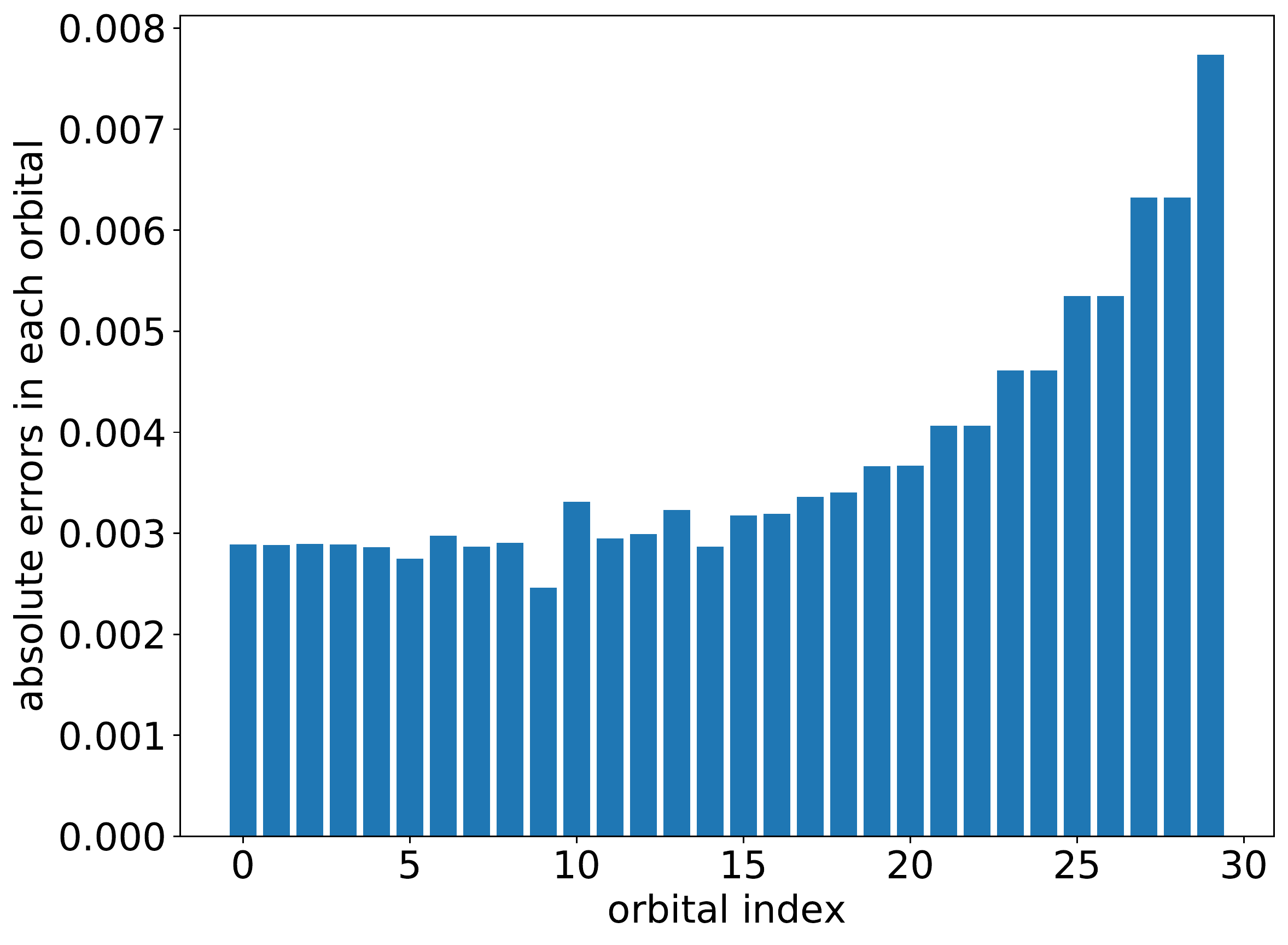} }
\subfloat[$N_e = 20$]{
\includegraphics[scale=0.25]{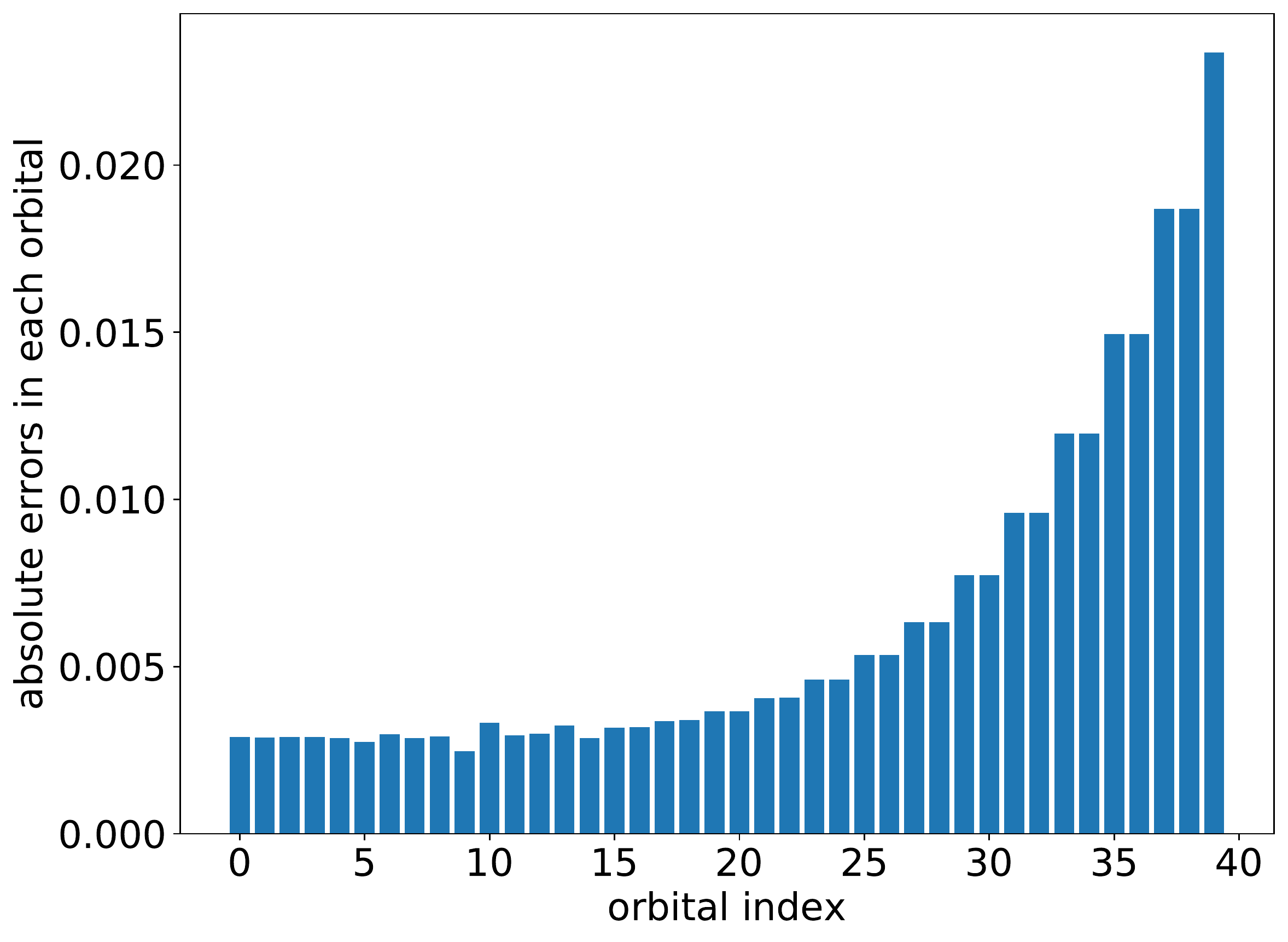} }
\\
\subfloat[$N_e = 30$]{
\includegraphics[scale=0.25]{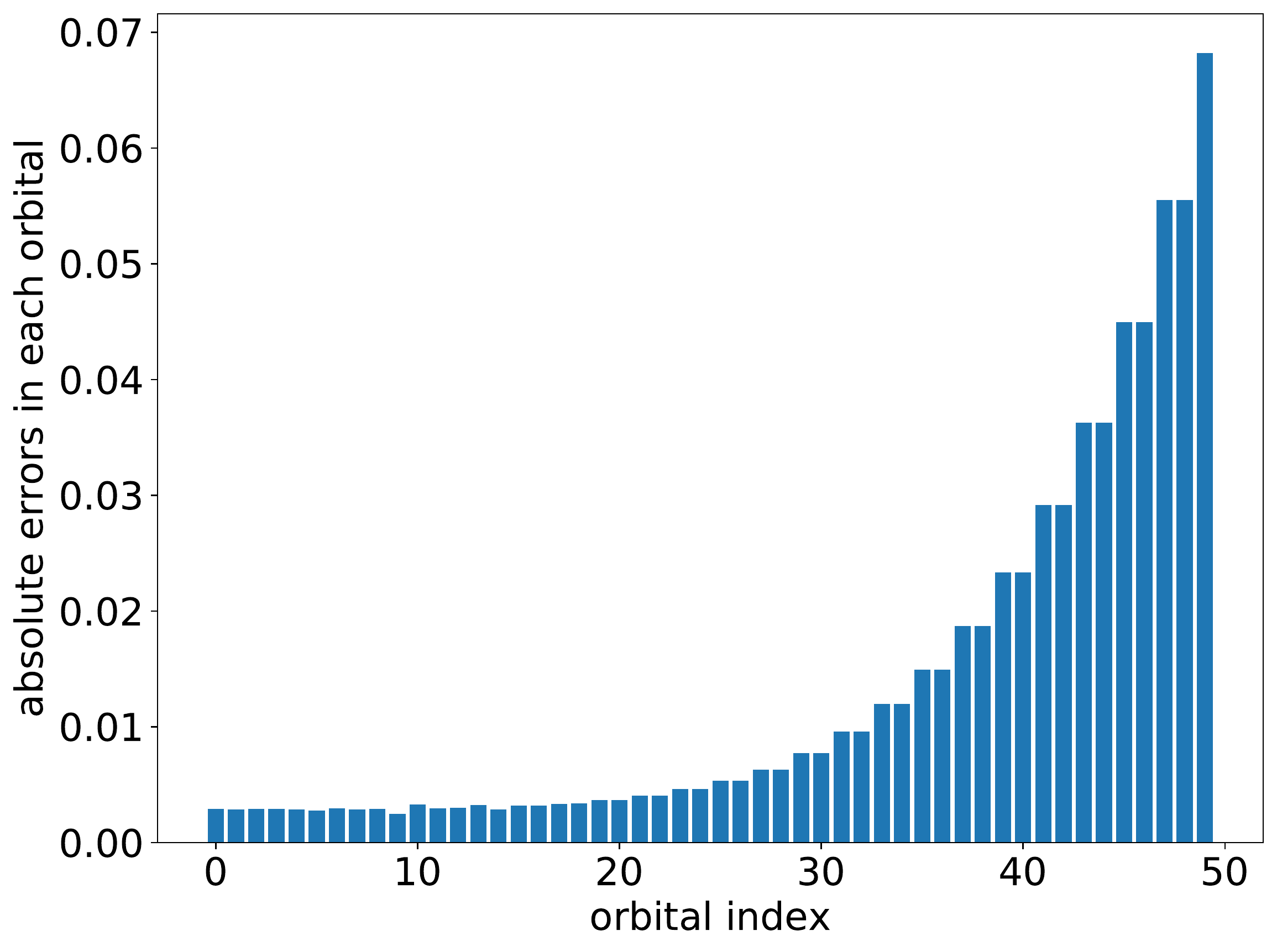} }
\subfloat[$N_e = 40$]{
\includegraphics[scale=0.25]{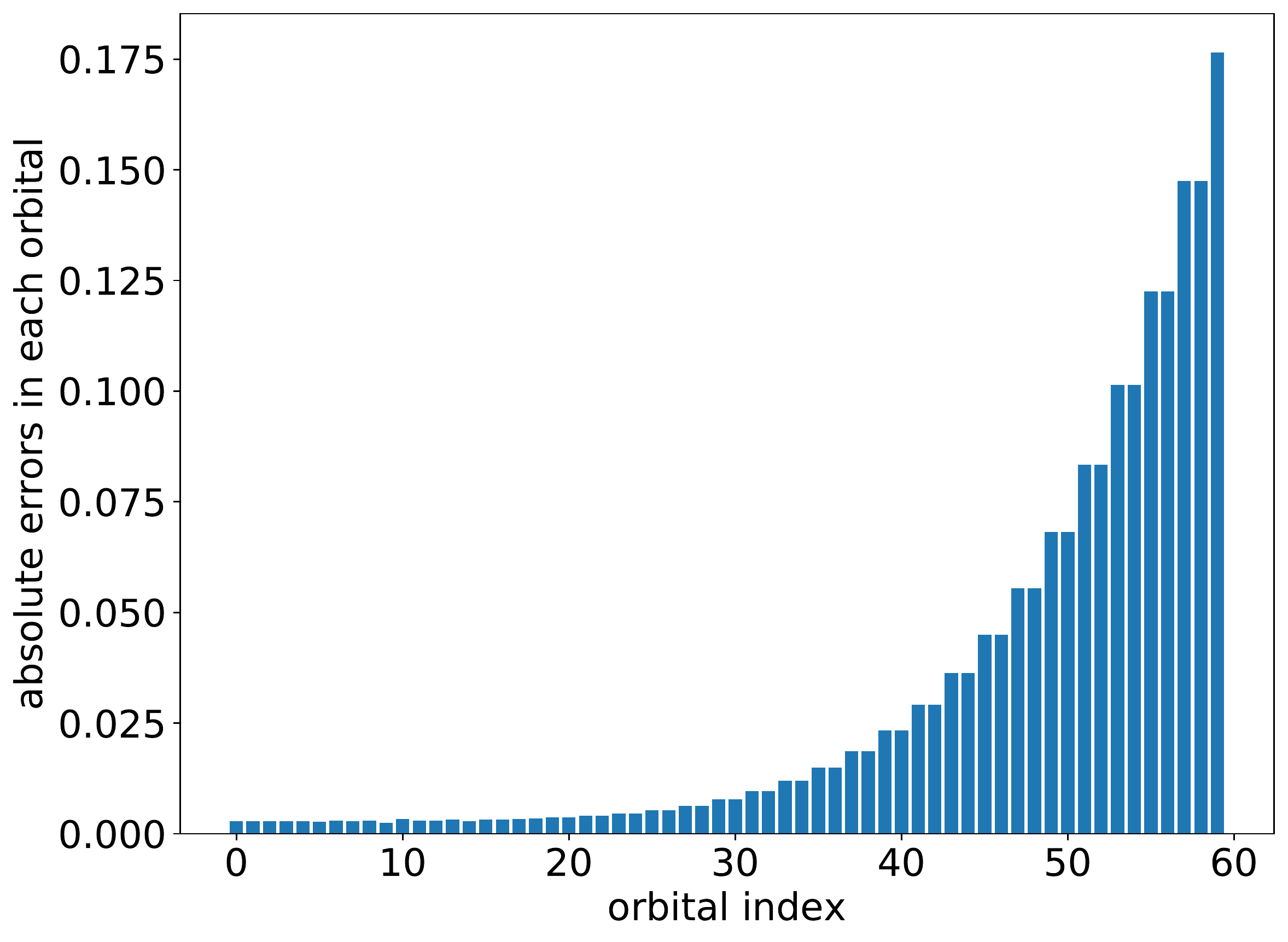} }
\\
\subfloat[$N_e = 50$]{
\includegraphics[scale=0.25]{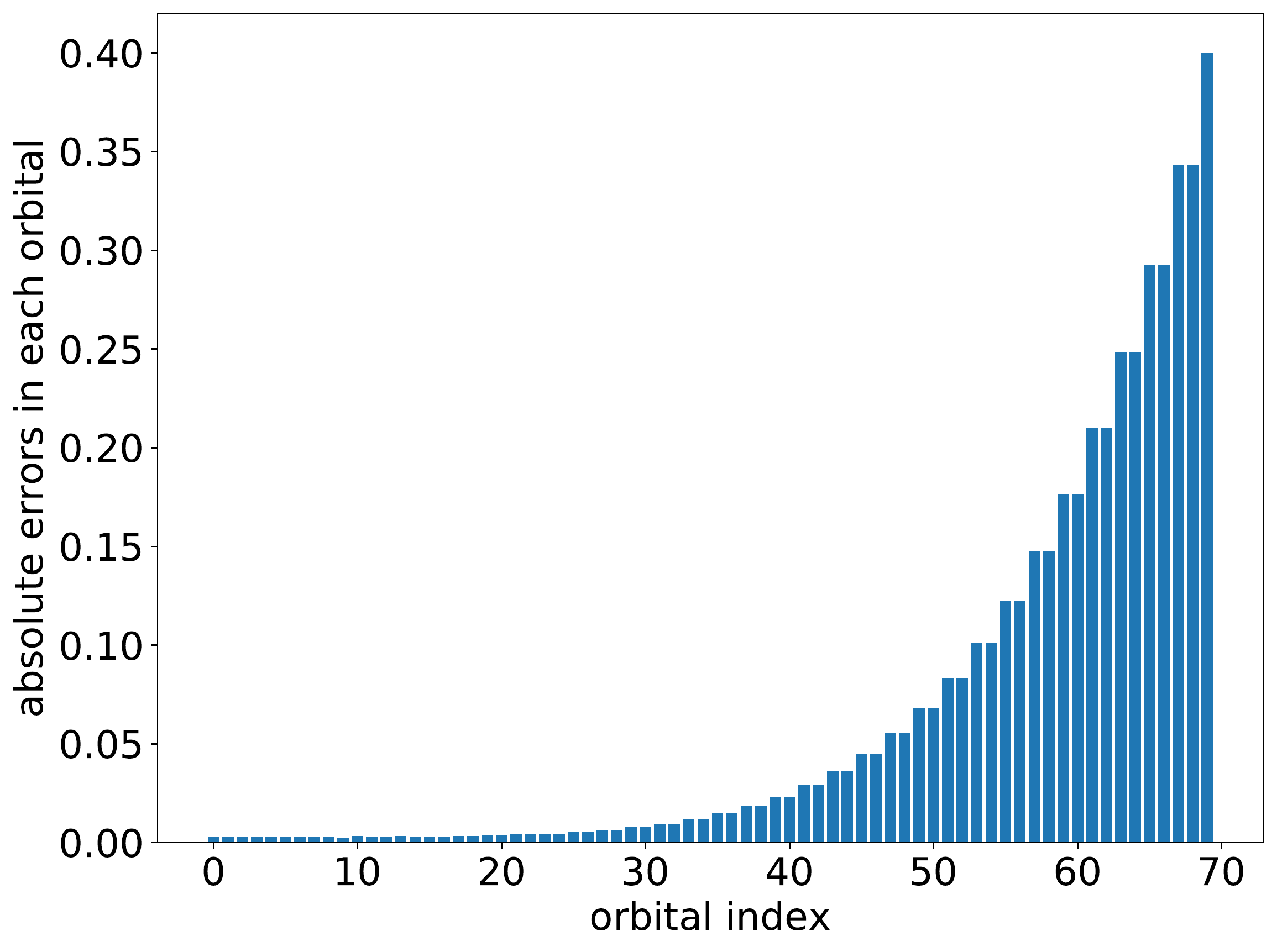} }
\subfloat[$N_e = 60$]{
\includegraphics[scale=0.25]{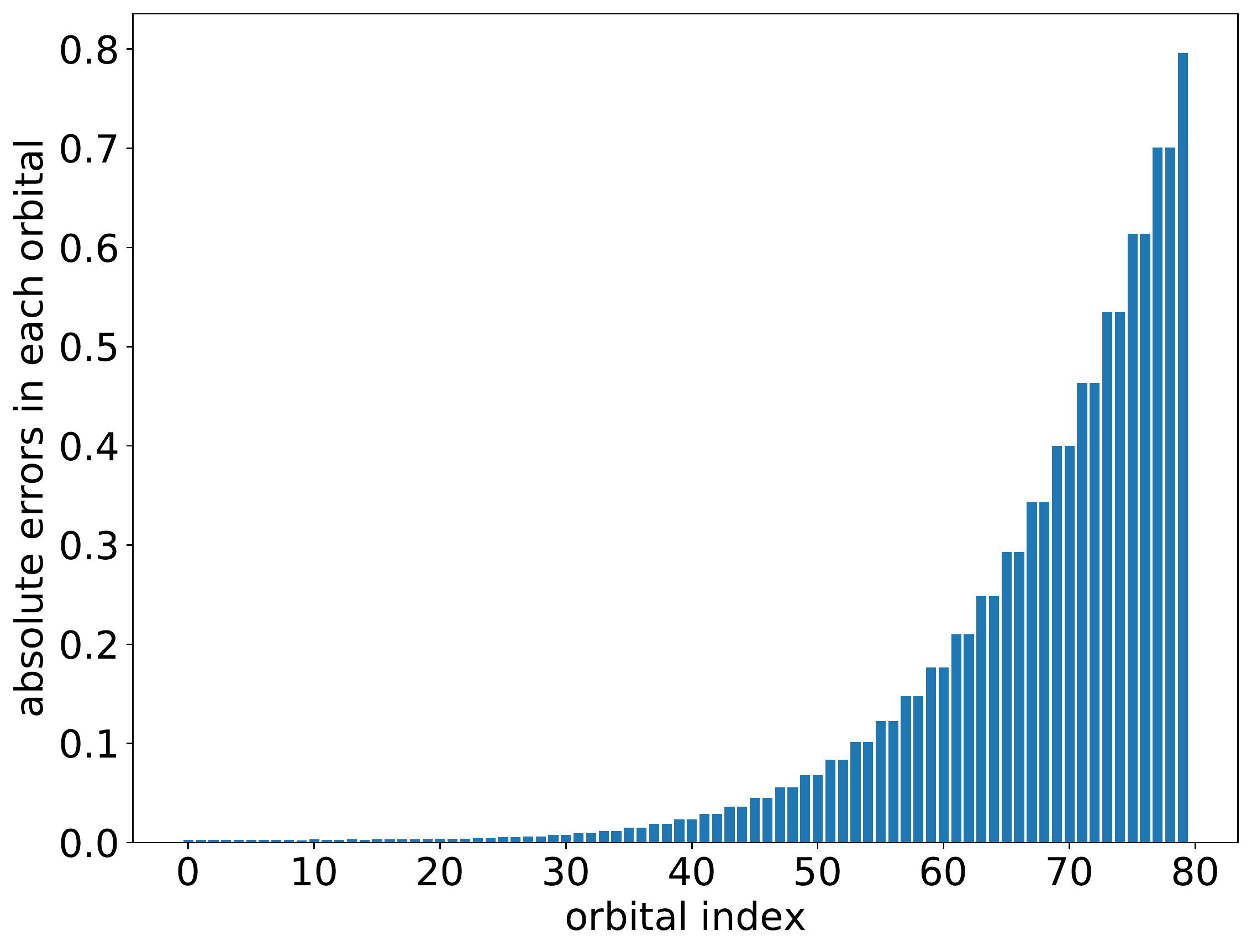} }
\caption{Plot of the error histogram of all orbitals for various $N_e$ in the Schr\"odinger dynamics. 
}
\label{fig:his_schd}
\end{figure}

Finally, we demonstrate that the PT dynamics remains equally effective in the nonlinear regime. The rt-TDDFT Hamiltonian takes the following general form
\begin{equation}
H(t, \rho(t)) = -\frac{1}{2}\Delta + V_{\text{ext}}(x,t) + V_\text{Hxc}[\rho(t)] + V_\text{X}[\rho(t)],
\label{eqn:ham_rttddft}
\end{equation}
where $V_\text{ext}$ represents the electron-ion interaction and when the external field changes with respect to time, $V_\text{ext}$ may also depend on time $t$. $V_\text{Hxc}$ is the Hartree and local exchange-correlation contribution and depends only on the diagonal part of $\rho(t)$, and $V_\text{X}$ is the Fock exchange operator depending on the entire $\rho(t)$. More specifically, $V_\text{X}$ is an integral operator defined by
\[
[(V_\text{X}[\rho])\phi](x) =
 -\int K(x,y) \rho(x,y) \phi(y) \,dy
\]
with the kernel $K(x,y)$ represents the electron-electron interaction. 

Following \cref{eqn:ham_rttddft}, we consider the following model problem
\begin{equation} \label{eqn:num_ex_H_nl}
H(t,\rho) = -\frac{1}{2}\Delta + V(x) + W(x,t) + U[\rho],
\end{equation}
where $V = x^2$, and $W$ is as defined in \eqref{eqn:Wt} and the nonlinear term $U[\rho]$ models
$V_X[\rho]$ with the Yukawa kernel
$$K(x,y) = \frac{2\pi}{\kappa \epsilon_0} e^{- \kappa |x-y|}.$$ 
Note that as $\kappa \to 0$, the Yukawa kernel approaches to the Coulomb interaction that diverges in one dimension and hence is typically used in place of the bare Coulomb interaction for one-dimensional problems. The parameters are chosen as $\epsilon_0 = 100$ and $\kappa = 0.01 $ so that the range of the electrostatic interaction is sufficiently long. Here $\mu = 148.99$ so that $N_e = 60$ and we choose $N = 80$. We simulate the system using PT-IM and SD-IM up to \ $T_{\text{final}} = 0.5$ and compare the relative errors. As shown in \cref{fig:error_dt_nl}, the errors from PT-IM are significantly smaller. A comparison of the dipole moment is presented in \cref{fig:dip_nl}. We also compute the dipole moment using $h = 0.01$ and compare the results with the reference solution obtained using SD-IM with a very fine time step $h = 0.0001$. It can be seen that the result using the PT dynamics agrees well with the reference, which is not the case for the Schr\"odinger dynamics with the same step size.

\begin{figure}[htbp]
\centering
\subfloat[Relative Errors]{
\includegraphics[scale=0.35]{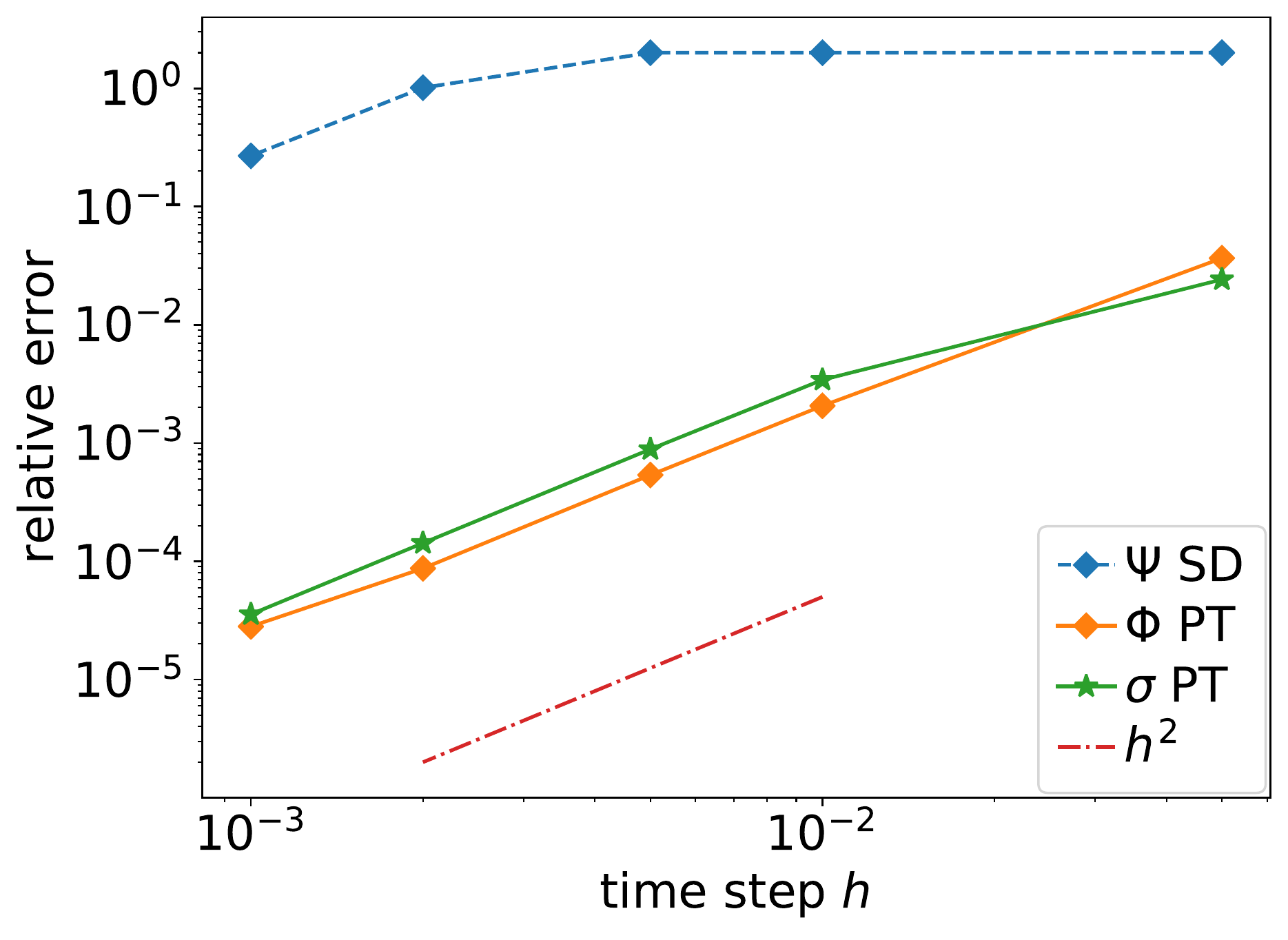} \label{fig:error_dt_nl}}
\subfloat[Dipole moment]{
\includegraphics[scale=0.35]{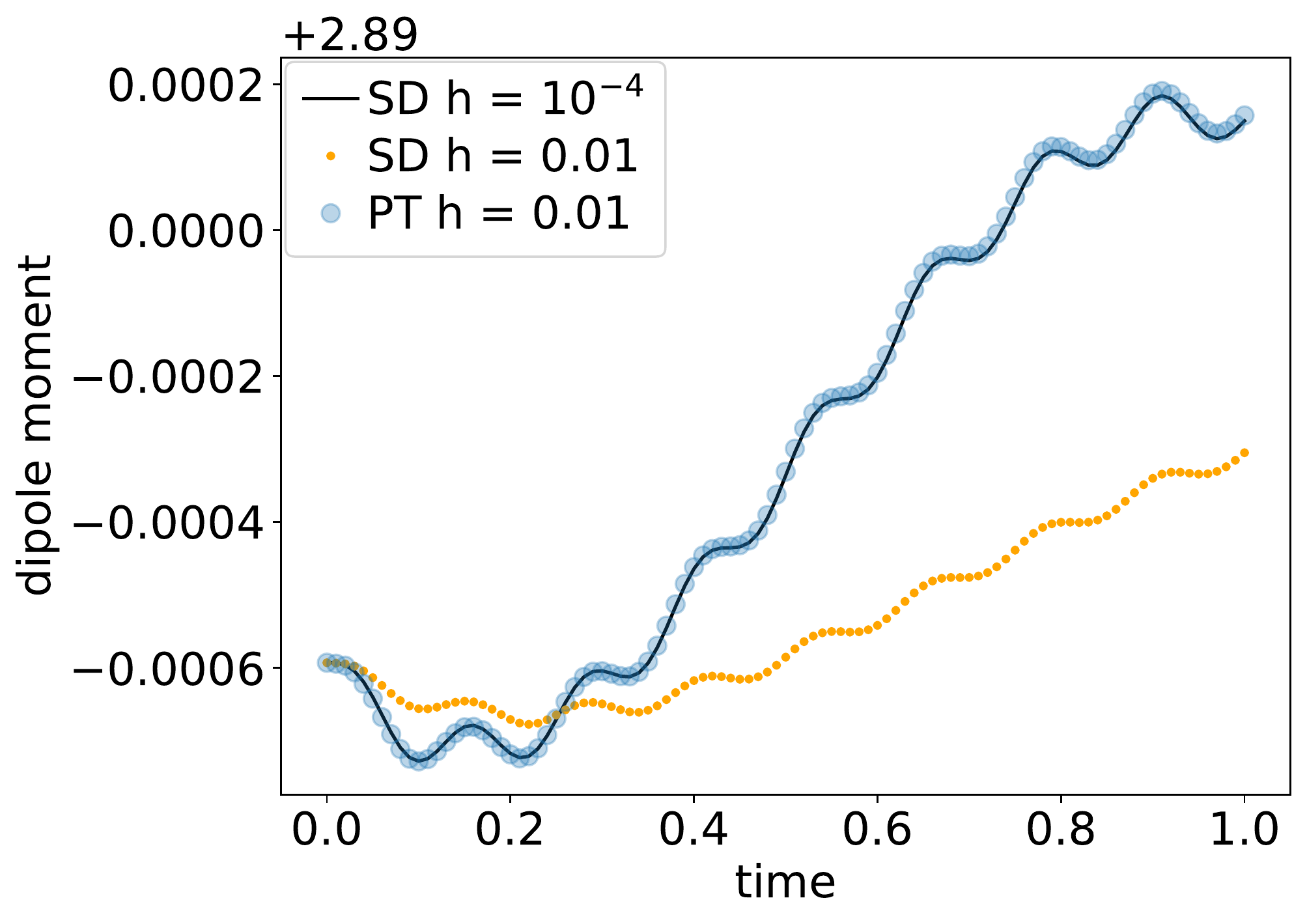} \label{fig:dip_nl}}
\caption{Left Panel: The log-log plot of the relative errors of $\Phi$, $\sigma$ in the model nonlinear rt-TDDFT calculation with \cref{eqn:num_ex_H_nl}, computed via PT-IM and SD-IM. Right Panel: a comparison of the dipole moment. 
} 
\end{figure}

\section{Conclusion}
In this paper, we have introduced the PT dynamics for mixed quantum states, which generalizes the PT dynamics for pure states presented in \cite{AnLin2020}. Both the PT and Schr\"odinger dynamics employ the low-rank structure of the density matrix, and can produce the same density matrix and all derived physical observables (such as dipole moments) as those from the von Neumann equation in the continuous time limit. The PT dynamics differ from the Schr\"odinger dynamics in terms of the choice of the gauge. In particular, the PT gauge yields the slowest possible dynamics for the wavefunctions. This allows us to significantly increase the time step size in the numerical simulation while maintaining accuracy.

As a concrete example, we propose the parallel transport-implicit midpoint (PT-IM) scheme, which is an implicit method suitable for treating Hamiltonians with a large spectral radius. It also preserves certain trace conditions and the orthogonality of the wavefunctions. We establish a new error bound for the PT dynamics, where all terms in the error bounds involve either the commutator of the Hamiltonian (and its derivatives) and the density matrix (or the associated spectral projector). As a comparison, the error analysis of the Schr\"odinger dynamics is also provided, which does not exhibit such commutator scaling. This new error bound, together with various numerical experiments, justifies the advantage of the PT dynamics for the general mixed states, where the dynamics can be nonlinear, and beyond the near adiabatic regime. Implementation of PT-IM for rt-TDDFT calculations of large scale real metallic systems is under progress.

\section*{Acknowledgment}

This work was partially supported by the Department of Energy under Grant No. DE-SC0017867 and No. FWP-NQISCCAWL (D.A., L.L.), and by the National Science Foundation under No. DMS-1652330, OMA-2016245 (L.L.).

\appendix

\section{Alternative derivation of the parallel transport dynamics using the tangent space formulation}\label{append:pt_tangent}

In this section, we provide an alternative derivation of the PT dynamics using the tangent space formulation. The presentation of this derivation consists of three parts: we first write down \cref{eqn:rhot_rotate_mixed} as well as the domain of the low-rank factors. We then derive the equation of motion of the low-rank factors of the rank-$N$ density matrix. Finally, we introduce the optimal gauge (i.e. PT gauge).

Let $\rho$ be of rank $N$ ($N_e \le N< N_g$). Recall \cref{eqn:rhot_rotate_mixed}: \begin{equation} \label{eqn:rhot_rotate_mixed_appendix}
\rho(t)=\Phi(t)\sigma(t)\Phi^{\dag}(t),
\end{equation}
where $\sigma(t)$ is an $N \times N$ positive semidefinite Hermitian matrix, and $\Phi$ is an $N_g \times N$ complex-valued matrix that satisfies $\Phi^\dagger \Phi = I_N$, in other words, $\Phi$ belongs to the Stiefel manifold $\text{St}(N,N_g)$ defined as
$$\text{St}(N,N_g) = \{\Phi \in \mathbb{R}^{N_g\times N}: \Phi^\dagger \Phi = I_N\}.
 $$
As is explained in \cref{sec:pt_mixed}, this decomposition \eqref{eqn:rhot_rotate_mixed_appendix} in fact admits an equivalence relationship 
$(\Phi, \sigma) \equiv (\Phi U, U^\dagger \sigma U)$, namely, for any $N \times N$ unitary matrix $U$, 
$$\rho= \Phi\sigma \Phi^\dagger = (\Phi U) (U^\dagger\sigma U)(U^\dagger \Phi^\dagger).
$$

Consider the infinitesimal variation of the tangent map of $(\Phi,\sigma) \mapsto \Phi \sigma \Phi^\dagger$. As is shown in {\cite[Lemma~4]{CaoLu2018}}, the tangent space of the Stiefel manifold admits the parametrization $\I \omega \Phi$, where $\omega$ is a Hermitian matrix of size $N_g$. 
Denote
\begin{equation*} 
\dot \Phi =\I \omega \Phi, \quad \dot \sigma = \xi,
\end{equation*}
where $\xi$ is a traceless Hermitian matrix of size $N$. The infinitesimal variation of $\rho$ can thus be represented as
\begin{equation*}
\I [\omega, \rho] + \Phi \xi \Phi^\dagger = \Phi (\I[\Phi^\dagger \omega \Phi, \sigma] + \xi) \Phi^\dagger.
\end{equation*}
We then project $\dot \rho$ onto this tangent space by minimizing the distance between them, namely, 
$$\min \,  \left\Vert-\I[H(t, \rho(t)), \rho(t)] - \I [\omega, \rho] - \Phi \xi \Phi^\dagger\right\Vert_F.$$
Hereafter, we drop the $(t,\rho)$ in $H$ for simplicity. The two stationary conditions in $\omega$ and $ \xi$ read
\begin{align}
\left[-\I [H,\rho]-\I[\omega, \rho] - \Phi \xi \Phi^\dagger, \rho\right] = 0 ,\label{eq:stationary_omega} \\
\Phi^\dagger (-\I[H,\rho] - \I[\omega, \rho] - \Phi \xi \Phi^\dagger )\Phi = \lambda I_N,\label{eq:stationary_xi} 
\end{align}
where $\lambda$ is the Lagrange multiplier introduced to satisfy the traceless condition of $\xi$. By taking trace on both sides of \eqref{eq:stationary_xi}, we find that
$$\lambda = \frac{1}{N} \Tr\left[
\Phi^\dagger \left(-\I[H,\rho] - \I [\omega, \rho] \right) \Phi
\right] =0,$$
because for any $X$,
\[
\Tr(\Phi^\dagger [X,\rho] \Phi) = \Tr(\Phi^\dagger X \Phi \sigma \Phi^\dagger \Phi - \Phi^\dagger \Phi \sigma \Phi^\dagger X \Phi)
=  \Tr(\Phi^\dagger X \Phi \sigma - \sigma \Phi^\dagger X \Phi)
= 0.
\]
Define projection operators
$$P : = \Phi \Phi^\dagger, \quad Q = I -P,$$
and one can express $\Phi \xi \Phi^\dagger$ using \eqref{eq:stationary_xi} as
\begin{equation} \label{eq:xi}
\xi  = \Phi^\dagger \left(-\I [H,\rho] - \I [\omega, \rho] \right) \Phi,
\quad \Phi \xi \Phi^\dagger = P \left(-\I[H,\rho] - \I [\omega, \rho] \right) P.
\end{equation}
Together with \cref{eq:stationary_omega}, we obtain
$$Q\left(-\I H\rho - \I \omega \rho \right) \rho =\rho \left(\I \rho H + \I\rho \omega \right) Q.$$
Note that $P\rho = \rho$ and hence the left-hand side stays in the range of $Q$ while the right-hand side remains in the range of $P$. By orthogonality, both sides of the equation vanish, which imposes some constraints on $Q\omega P$ and $P\omega Q$. To be specific, one has
$$Q \omega \rho^2 = -Q H\rho^2 \quad \iff \quad Q \omega \Phi\sigma^2 \Phi^\dagger = - Q H \Phi\sigma^2 \Phi^\dagger.
$$
Right multiply $\Phi(\sigma^{-1})^2 \Phi^\dagger$, one finds that
$Q \omega P = - QHP.$
Similarly, we obtain $P\omega Q = -  PHQ.$
The general solution of this system of matrix equations is 
\begin{equation*} 
\omega = -  Q H P -  PHQ - PGP - QGQ,
\end{equation*}
where $G$ is any Hermitian matrix due to the hermiticity of $\omega$. Since $\dot \Phi = \I\omega \Phi$, the equation for $\Phi$ can be written as
\begin{align*}
\I \dot \Phi & = - \omega \Phi= Q H P \Phi+ PHQ \Phi +  PGP \Phi + QGQ \Phi\\
& = QH \Phi +  PG \Phi,
\end{align*}
where the fact that $P\Phi = \Phi$ and $Q\Phi = 0$ is used. For the equation of $\xi$, \eqref{eq:xi} yields
\begin{align*}
\I \dot \sigma & = \I \xi = \I  \Phi^\dagger \left(-\I[H,\rho] - \I [\omega, \rho] \right) \Phi \\
& = [\Phi^\dagger H \Phi,\sigma] - [\Phi^\dagger G \Phi, \sigma].
\end{align*}
Finally, we arrive at the dynamics for $\Phi$ and $\sigma$ in a closed form
\begin{align} \label{eq:dyn_w_gauge}
\I \partial_t \Phi &=   (I- \Phi \Phi^\dagger)H(t,\Phi \sigma \Phi^\dagger) \Phi + \eps \Phi \Phi^\dagger G \Phi, \\
\I \partial_t\sigma & = \left[\Phi^\dagger\left(H(t,\Phi \sigma \Phi^\dagger)- G\right)\Phi, \sigma\right],
\end{align}
where the Hermitian matrix $G$ is an extra degree of freedom to be chosen.

The next step is to find the optimal choice of $G$ such that the dynamics of $\Phi$ changes the slowest, i.e. to find $G$ such that 
\begin{equation} \label{eq:opt}
\min\|\dot \Phi\|_F^2 = \min \Tr(\dot \Phi^\dagger \dot \Phi) . 
\end{equation}
The norm can be split into two parts 
\begin{align*}
\|\dot \Phi\|_F^2 &= \|P \dot \Phi\|_F^2 + \| Q\dot \Phi\|_F^2 \\
& = \|\Phi \Phi^\dagger G \Phi\|_F^2 + \|QH \Phi\|_F^2.
\end{align*}
The Frobenius norm of the second term reads 
\[
\Tr(\Phi^\dagger H^\dagger Q^\dagger Q H \Phi) = \Tr(Q^\dagger Q H \Phi \Phi^\dagger H^\dagger ) 
 = \Tr(Q^\dagger Q H(t, \rho) P H^\dagger(t, \rho) ) ,
\]
which is independent of the gauge choice. Therefore, to optimize \eqref{eq:opt} one can choose $G = 0$. Now we arrive at the parallel transport dynamics
\begin{align} \label{eq:pt_main}
\I \partial_t \Phi &=   (I - \Phi \Phi^\dagger)H(t,\Phi \sigma \Phi^\dagger) \Phi , \\
\I \partial_t\sigma & = \left[\Phi^\dagger H(t,\Phi \sigma \Phi^\dagger)\Phi, \sigma\right], \notag
\end{align}
which is equivalent to the PT dynamics \eqref{eqn:pt_mix} derived in \cref{sec:pt_mixed}.

It is worth pointing out that the Schr\"odinger gauge in fact corresponds to the choice $G =H$ in \eqref{eq:dyn_w_gauge} that gives rise to
\begin{align}\label{eq:sch_main}
\I \partial_t \Phi &=   H(t,\Phi\sigma \Phi^\dagger)\Phi, \\
\partial_t\sigma & = 0.\notag
\end{align}
This immediately implies that the number of occupied orbitals remains unchanged throughout the evolution, which verifies the validity of the solution in the Schr\"odinger dynamics in \cref{eqn:rho_def_schd}.

\section{Anderson's mixing method for solving nonlinear equations}\label{sec:anderson}

Let $x_k$ be the value of $(\Phi_{n+1},\sigma_{n+1})$ at the $k$-th iteration, and define
\[
r_k=x_k-T(x_k), \quad s_k=x_k-x_{k-1}, \quad y_k=r_k-r_{k-1},
\]
then Anderson's mixing method obtains an update $x_{k+1}$ according to
\[
x_{k+1}=x_k-C_kr_k.
\]
The matrix $C_k$ can be viewed as an approximation to the Jacobian of the mapping $x\mapsto x-T(x)$. Define 
\[
S_{k}=\left(s_{k}, s_{k-1}, \ldots, s_{k-\ell}\right), \quad Y_{k}=\left(y_{k}, y_{k-1}, \ldots, y_{k-\ell}\right),
\]
where $\ell$ is called the mixing dimension. The matrix $C_k$ is obtained by solving the minimization problem
\begin{equation}
\begin{split}
\min_{C} & \quad \frac{1}{2}\left\|C-C_{0}\right\|_{F}^{2} \\
\text { s.t. } & \quad S_{k}=C Y_{k},
\end{split}
\label{eqn:min_anderson}
\end{equation}
where $C_0$ is an initial guess to the Jacobian. In the absence of further information, we may choose $C_0=\alpha I$ for some constant $\alpha$. The solution to the minimization problem
\[
C_{k}=C_{0}+\left(S_{k}-C_{k-1} Y_{k}\right) Y_{k}^{+},
\]
where $Y_{k}^{+}=\left(Y_{k}^{\dag} Y_{k}\right)^{-1} Y_{k}^{\dag}$ is the Moore-Penrose pseudoinverse. 
Therefore, the update rule for $x_{k+1}$ is
\begin{equation}
x_{k+1}=x_{k}-C_{0}\left(I-Y_{k} Y_{k}^{+}\right) r_{k}-S_{k} Y_{k}^{+} r_{k}.
\label{eqn:anderson_update}
\end{equation}

For the initial step, we may simply take $x_0$ to be the vectorized form of $(\Phi_{n},\sigma_n)$ at the previous time step. In order to assess the computational complexity, we assume that the application of the Hermitian matrix to wavefunctions in the form of $H\Phi$ can be obtained at cost $\Or(N_gN^2)$ without explicitly forming $H$. Then the cost of updating $x_k$ to $x_{k+1}$ is $\Or(N_g N^2+N^3)$, and thus governed by $\Or(N_g N^2)$ under the assumption $N_g \gg N$.

\bibliographystyle{abbrv}
\bibliography{pt,pt_di}

\end{document}